\documentclass[12pt]{article}
\usepackage[a4paper]{geometry}
\usepackage{amssymb,amsmath,amsthm}
\usepackage{graphicx,cite,color}
\usepackage{longtable,pdflscape,booktabs,caption,multicol}
\usepackage[colorlinks=true,citecolor=black,linkcolor=black,urlcolor=blue]{hyperref}

\theoremstyle{plain}
\newtheorem{theorem}{Theorem}
\newtheorem{lemma}[theorem]{Lemma}

\newtheorem{conjecture}[theorem]{Conjecture}

\theoremstyle{definition}

\theoremstyle{remark}

\DeclareMathOperator{\Circ}{Circ}

\usepackage{float}
\usepackage{tikz}
\usetikzlibrary{arrows}
\usetikzlibrary{arrows.meta}
\usetikzlibrary{chains}
\usetikzlibrary{positioning}
\usetikzlibrary{automata,positioning,calc}
\usetikzlibrary{decorations}
\usetikzlibrary{decorations.shapes}
\usetikzlibrary{decorations.markings}
\tikzset{
    edge/.style={-{Latex[scale=1.7]}},
    dedge/.style={{Latex[scale=1.7]}-{Latex[scale=1.7]}},
}

\title{Two families of circulant nut graphs}

\author{Ivan Damnjanovi\'c\thanks{The author is supported by Diffine LLC.}\\
\small University of Ni\v s, Faculty of Electronic Engineering\\[-0.4ex]
\small\tt ivan.damnjanovic@elfak.ni.ac.rs\\
\small Diffine LLC\\[-0.4ex]
\small\tt ivan@diffine.com}

\begin{document}

\maketitle

\begin{abstract}
A circulant nut graph is a non-trivial simple graph whose adjacency matrix is a circulant matrix of nullity one such that its non-zero null space vectors have no zero elements. The study of circulant nut graphs was originally initiated by Ba\v{s}i\'c et al.\ [Art Discrete Appl.\ Math.\ \textbf{5(2)} (2021) \#P2.01], where a conjecture was made regarding the existence of all the possible pairs $(n, d)$ for which there exists a $d$-regular circulant nut graph of order $n$. Later on, it was proved by Damnjanovi\'c and Stevanovi\'c [Linear Algebra Appl.\ \textbf{633} (2022) 127--151] that for each odd $t \ge 3$ such that $t\not\equiv_{10}1$ and $t\not\equiv_{18}15$, the $4t$-regular circulant graph of order $n$ with the generator set $\{ 1, 2, 3, \ldots, 2t+1 \} \setminus \{t\})$ must necessarily be a nut graph for each even $n \ge 4t + 4$. In this paper, we extend these results by constructing two families of circulant nut graphs. The first family comprises the $4t$-regular circulant graphs of order $n$ which correspond to the generator sets $\{1, 2, \ldots, t-1\} \cup \left\{\frac{n}{4}, \frac{n}{4} + 1 \right\} \cup \left\{\frac{n}{2} - (t-1), \ldots, \frac{n}{2} - 2, \frac{n}{2} - 1 \right\}$, for each odd $t \in \mathbb{N}$ and $n \ge 4t + 4$ divisible by four. The second family consists of the $4t$-regular circulant graphs of order $n$ which correspond to the generator sets $\{1, 2, \ldots, t-1\} \cup \left\{\frac{n+2}{4}, \frac{n+6}{4} \right\} \cup \left\{\frac{n}{2} - (t-1), \ldots, \frac{n}{2} - 2, \frac{n}{2}-1 \right\}$, for each $t \in \mathbb{N}$ and $n \ge 4t + 6$ such that $n \equiv_{4} 2$. We prove that all of the graphs which belong to these families are indeed nut graphs, thereby fully resolving the $4t$-regular circulant nut graph order--degree existence problem whenever $t$ is odd and partially solving this problem for even values of $t$ as well.

\bigskip\noindent
{\bf Mathematics Subject Classification:} 05C50, 12D05, 13P05, 11C08. \\
{\bf Keywords:} Circulant graphs; Nut graphs; Graph spectra; Graph eigenvalues; Cyclotomic polynomials.
\end{abstract}

\section{Introduction}

In this paper we will consider all graphs to be undirected, finite, simple and non-null. Thus, every graph will have at least one vertex and there shall be no loops or multiple edges. Also, for convenience, we will take that each graph of order $n$ has the vertex set $\{0, 1, 2, \ldots, n-1\}$.

A graph $G$ is considered to be a circulant graph if its adjacency matrix $A$ has the form
\[
A = \begin{bmatrix}
    a_0 & a_1 & a_2 & \cdots & a_{n-1}\\
    a_{n-1} & a_0 & a_1 & \cdots & a_{n-2}\\
    a_{n-2} & a_{n-1} & a_0 & \cdots & a_{n-3}\\
    \vdots & \vdots & \vdots & \ddots & \vdots\\
    a_1 & a_2 & a_3 & \dots & a_0
\end{bmatrix}.
\]
Here, we clearly have $a_0 = 0$, as well as $a_j = a_{n-j}$ for all $j = \overline{1, n}$. A concise way of describing a circulant graph is by taking into consideration the set of all the values $1 \le j \le \frac{n}{2}$ such that $a_j = a_{n-j} = 1$. We shall refer to this set as the generator set of a circulant graph and we will use $\mathrm{Circ}(n, S)$ to denote the circulant graph of order $n$ whose generator set is $S$.

A nut graph is a non-trivial graph whose adjacency matrix has nullity one and is such that its non-zero null space vectors have no zero elements, as first described by Sciriha in \cite{Sciriha}. Bearing this in mind, a circulant nut graph is simply a nut graph whose adjacency matrix additionally represents a circulant matrix. Ba\v{s}i\'c et al.\ \cite{Basic} initiated the study of these graphs by providing several results, alongside the following two conjectures.

\begin{conjecture}[Ba\v{s}i\'c et al.\ \cite{Basic}]\label{basic_conjecture_1}
For every even $n,\, n \ge 16$, there exists a circulant nut graph $\mathrm{Circ}(n, \{s_1, s_2, s_3, s_4, s_5, s_6 \})$ of degree $12$.
\end{conjecture}

\begin{conjecture}[Ba\v{s}i\'c et al.\ \cite{Basic}]
For every $d$, where $d \equiv_4 0$, and for every even $n,\, n \ge d+4$, there exists a circulant nut graph $\mathrm{Circ}(n, \{a_1, a_2, a_3, \ldots, a_{\frac{d}{2}}\})$ of degree $d$.
\end{conjecture}

Later on, Damnjanovi\'c and Stevanovi\'c \cite[Proposition 8]{Damnjanovic} proved Conjecture~\ref{basic_conjecture_1} by showing that $\mathrm{Circ}(n, \{1, 2, 4, 5, 6, 7\})$ is a $12$-regular nut graph for each even $n \ge 16$. In fact, it was shown that for infinitely many odd values of $t$ it is possible to construct similar families of $4t$-regular circulant nut graphs, as demonstrated in the next theorem.

\begin{theorem}[Damnjanovi\'c and Stevanovi\'c \cite{Damnjanovic}]\label{damnjanovic_theorem_1}
For each odd $t \ge 3$ such that $t\not\equiv_{10} 1$ and $t\not\equiv_{18} 15$,
the circulant graph $\Circ(n, \{1, 2, 3, \ldots, 2t+1\} \setminus \{ t \})$ is a nut graph for each even $n \ge 4t+4$.
\end{theorem}

The primary motivation behind this paper is to improve the currently existing results regarding the circulant nut graph order--degree existence problem. In other words, the goal is to answer the question whether there exists a $d$-regular circulant nut graph of order $n$, for any such pair $(n, d) \in \mathbb{N}^2$ of integers. As shown by Damnjanovi\'c and Stevanovi\'c \cite[Lemma 6]{Damnjanovic}, the order of each circulant nut graph must be even, while its degree must be divisible by four. It is clear that a $4t$-regular circulant nut graph cannot have an order of $4t+2$, since the adjacency matrix of such a graph would obviously have a nullity greater than one, due to the fact that its first $\frac{n}{2}$ rows would match its following $\frac{n}{2}$ rows, respectively.

Furthermore, if $t$ is even, then the order of a $4t$-regular circulant nut graph cannot be below $4t+6$, as demonstrated by Damnjanovi\'c and Stevanovi\'c \cite[Lemma~18]{Damnjanovic}. Bearing this in mind, we deduct that for a given $t \in \mathbb{N}$, all the $4t$-regular circulant nut graphs must have an even order $n$ such that
\begin{itemize}
    \item $n \ge 4t + 4$ if $2 \nmid t$;
    \item $n \ge 4t + 6$ if $2 \mid t$.
\end{itemize}

In this paper, we present the following two families of circulant graphs:
\begin{alignat*}{2}
    \mathcal{D}'_{t, n} &= \Circ(n, \mathcal{S}'_{t, n}), \quad && \mbox{for each $2 \nmid t$ and $4 \mid n$ such that $n \ge 4t + 4$},\\
    \mathcal{D}''_{t, n} &= \Circ(n, \mathcal{S}''_{t, n}), \quad && \mbox{for each $t\in\mathbb{N}$ and $n \equiv_4 2$ such that $n \ge 4t + 6$},
\end{alignat*}
where
\begin{align*}
    \mathcal{S}'_{t, n} &= \{1, 2, \ldots, t-1\} \cup \left\{\frac{n}{4}, \frac{n}{4} + 1 \right\} \cup \left\{\frac{n}{2} - (t-1), \ldots, \frac{n}{2} - 2, \frac{n}{2} - 1 \right\} ,\\
    \mathcal{S}''_{t, n} &= \{1, 2, \ldots, t-1\} \cup \left\{\frac{n+2}{4}, \frac{n+6}{4} \right\} \cup \left\{\frac{n}{2} - (t-1), \ldots, \frac{n}{2} - 2, \frac{n}{2}-1 \right\} .
\end{align*}
Subsequently, we prove that all of the graphs which belong to these families are nut graphs, thereby showing that for each odd $t$, there does exist a $4t$-regular circulant nut graph of order $n$, for any even $n \ge 4t + 4$. This observation fully resolves the circulant nut graph order--degree existence problem for $4t$-regular graphs whenever $t$ is odd.

When it comes to the even values of $t$, we get a partial resolution of the existence problem, given the fact that the said constructions do cover the case when $n \equiv_4 2$. However, the case when $n$ is divisible by four is not yet resolved, hence these graphs remain to be further inspected in the future.

The remainder of this paper is structured as follows. In Section~\ref{preliminaries} we preview certain theoretical facts regarding the circulant matrices, circulant nut graphs and cyclotomic polynomials which are required in order to successfully prove the two main theorems. Subsequently, Section~\ref{d1_family} displays a full mathematical proof of the fact that each $\mathcal{D}'_{t, n}$ circulant graph is indeed a nut graph. Afterwards, Section~\ref{d2_family} uses a similar strategy in order to prove that every $\mathcal{D}''_{t, n}$ graph must be a circulant nut graph as well. Finally, Section~\ref{conclusion} provides a brief conclusion to all the obtained results and discloses a new conjecture to be solved later on.

\section{Preliminaries}\label{preliminaries}

First of all, it is worth pointing out that the generator sets $\mathcal{S}'_{t, n}$ and $\mathcal{S}''_{t, n}$ are always well defined. It is straightforward to check that
\[
    1 < 2 < \cdots < t-1 < \frac{n}{4} < \frac{n}{4} + 1 < \frac{n}{2} - (t-1) < \cdots < \frac{n}{2} - 2 < \frac{n}{2} - 1 
\]
for each odd $t \ge 3$ and $n \ge 4t + 4$ that is divisible by four. Similarly, we have
\[
    1 < 2 < \cdots < t-1 < \frac{n+2}{4} < \frac{n+6}{4} < \frac{n}{2} - (t-1) < \cdots < \frac{n}{2} - 2 < \frac{n}{2} - 1
\]
for every $t \ge 2$ and $n \ge 4t + 6$ such that $n \equiv_4 2$. Finally, for $t = 1$ we have that
\[
    1 < \dfrac{n}{4} < \dfrac{n}{4} + 1 \le \dfrac{n}{2} - 1
\]
for every $n \ge 4t+4$ divisible by four, as well as
\[
    1 < \dfrac{n+2}{4} < \dfrac{n+6}{4} \le \dfrac{n}{2} - 1
\]
 for each $n \ge 4t + 6$ such that $n \equiv_4 2$.

It is known from elementary linear algebra theory (see, for example, \cite[Section~3.1]{Gray}) that the circulant matrix
\[
A = \begin{bmatrix}
    a_0 & a_1 & a_2 & \cdots & a_{n-1}\\
    a_{n-1} & a_0 & a_1 & \cdots & a_{n-2}\\
    a_{n-2} & a_{n-1} & a_0 & \cdots & a_{n-3}\\
    \vdots & \vdots & \vdots & \ddots & \vdots\\
    a_1 & a_2 & a_3 & \dots & a_0
\end{bmatrix}
\]
must have the eigenvalues
\[
    P(1), P(\omega), P(\omega^2), \ldots, P(\omega^{n-1}),
\]
where $\omega=e^{i \frac{2 \pi}{n}}$ is an $n$-th root of unity, and
\[
    P(x) = a_0 + a_1 x + a_2 x^2 + \cdots + a_{n-1} x^{n-1} .
\]

Starting from the aforementioned result, Damnjanovi\'c and Stevanovi\'c \cite{Damnjanovic} have managed to give the necessary and sufficient conditions for a circulant graph to be a nut graph in the form of the following lemma.
\newpage
\begin{lemma}[Damnjanovi\'c and Stevanovi\'c \cite{Damnjanovic}]\label{damnjanovic_lemma_1}
    Let $G = \Circ(n, S)$ where $n \ge 2$. The graph $G$ is a nut graph if and only if all of the following conditions hold:
    \begin{itemize}
        \item $2 \mid n$;
        \item $S$ consists of $t$ odd and $t$ even integers from $\left\{1, 2, 3, \ldots, \frac{n}{2} - 1 \right\}$, for some $t \ge 1$;
        \item $P(\omega^j)\neq 0$ for each $j\in\left\{1, 2, 3, \ldots, \frac{n}{2}-1 \right\}$.
    \end{itemize}
\end{lemma}

All of the graphs $\mathcal{D}'_{t, n}$ and $\mathcal{D}''_{t, n}$ clearly satisfy the first condition from Lemma~\ref{damnjanovic_lemma_1}. Besides that, it is easy to show that they necessarily satisfy the second condition as well. For $\mathcal{D}'_{t, n}$, it is enough to point out that both $\{1, 2, \ldots, t-1\}$ and $\left\{\frac{n}{2} - (t-1), \ldots, \frac{n}{2} - 2, \frac{n}{2} - 1 \right\}$ contain an even number of consecutive integers, due to the fact that $t-1$ is even. Since $\frac{n}{4}$ and $\frac{n}{4} + 1$ are surely of different parities, it follows that $\mathcal{S}'_{t, n}$ does contain equally many odd and even integers. For $\mathcal{D}''_{t, n}$, it is essential to notice that $j$ and $\frac{n}{2} - j$ must be of different parities, for each $j = \overline{1, t-1}$, due to the fact that $\frac{n}{2}$ is odd. Furthermore, $\frac{n+2}{4}$ and $\frac{n+6}{4}$ are certainly of different parities, hence we obtain that $\mathcal{S}''_{t, n}$ contains equally many odd and even integers, too.

Taking into consideration Lemma~\ref{damnjanovic_lemma_1}, it becomes apparent that in order to show that each graph $\mathcal{D}'_{t, n}$ and $\mathcal{D}''_{t, n}$ is a nut graph, it is sufficient to prove that it satisfies the third condition given in the lemma. In other words, it is enough to demonstrate that, for these graphs, the polynomial $P(x) \in \mathbb{Z}[x]$ has no $n$-th roots of unity among its roots, except potentially $-1$ or $1$.

It is clear that $\omega^{n-j} = \dfrac{1}{\omega^j}$ for each $j = \overline{1, n-1}$. Bearing this in mind, we quickly obtain that
\begin{equation}\label{polynomial_formula}
    P(\omega^j) = \left( \omega^{j s_0} + \frac{1}{\omega^{j s_0}} \right) + \left( \omega^{j s_1} + \frac{1}{\omega^{j s_1}} \right) + \cdots + \left( \omega^{j s_{k-1}} + \frac{1}{\omega^{j s_{k-1}}} \right)
\end{equation}
for an arbitrary circulant graph $G = \mathrm{Circ}(n, S)$, where $S = \{ s_0, s_1, s_2, \ldots, s_{k-1} \}$, provided all the generator set elements are lower than $\frac{n}{2}$. Formula (\ref{polynomial_formula}) will be heavily used throughout Sections \ref{d1_family} and \ref{d2_family} while proving the two main theorems of the paper.

Last but not least, it is crucial to point out that the cyclotomic polynomials shall play a key role in demonstrating whether or not certain polynomials of interest contain the given roots of unity among their roots. The cyclotomic polynomial $\Phi_b(x)$ can be defined for each $b \in \mathbb{N}$ via
\[
    \Phi_b(x) = \prod_{\zeta} (x - \zeta) ,
\]
where $\zeta$ ranges over the primitive $b$-th roots of unity. It is known that these polynomials have integer coefficients and that they are all irreducible in $\mathbb{Q}[x]$ (see, for example, \cite{cyclotomic}). Hence, an arbitrary polynomial in $\mathbb{Q}[x]$ has a primitive $b$-th root of unity among its roots if and only if it is divisible by $\Phi_b(x)$.

While inspecting whether certain integer polynomials are divisible by cyclotomic polynomials, we will strongly rely on the following theorem on the divisibility of lacunary polynomials by cyclotomic polynomials.
\begin{theorem}[Filaseta and Schinzel \cite{Filaseta}]\label{filaseta}
Let $P(x) \in \mathbb{Z}[x]$ have $N$ nonzero terms and let $\Phi_b(x) \mid P(x)$.
Suppose that $p_1, p_2, \dots, p_k$ are distinct primes such that
\[
    \sum_{j=1}^k (p_j-2) > N-2 .
\]
Let $e_j$ be the largest exponent such that $p_j^{e_j} \mid b$. Then for at least one $j$, $1 \le j \le k$, we have that $\Phi_{b'}(x)\mid P(x)$, where $b' = \dfrac{b}{p_j^{e_j}}$.
\end{theorem}

\section{\texorpdfstring{$\mathcal{D}'_{t, n}$}{D'(t, n)} family of circulant graphs}\label{d1_family}

In this section, we will formulate and provide a full mathematical proof of the first of the two central theorems given in the paper.
\begin{theorem}\label{main_theorem_1}
    For each odd $t \in \mathbb{N}$ and $n \ge 4t + 4$ such that $4 \mid n$, the circulant graph $\mathcal{D}'_{t, n}$ must be a $4t$-regular nut graph of order $n$.
\end{theorem}

In order to make the proof more concise, we will need a number of auxiliary lemmas. To start with, let $Q_t(x) \in \mathbb{Z}[x]$ and $R_t(x) \in \mathbb{Z}[x]$ be the following two integer polynomials
\begin{align*}
    Q_t(x) &= 2x^{2t-1} + x^{t+1} - x^t + x^{t-1} - x^{t-2} - 2,\\
    R_t(x) &= 2x^{2t-1} - x^{t+1} - 3x^t + 3x^{t-1} + x^{t-2} - 2,
\end{align*}
for each odd $t \ge 3$. It is clear that all of these polynomials must have exactly six non-zero terms, due to the fact that
\[
    0 < t-2 < t-1 < t < t+1 < 2t-1
\]
for each odd $t \ge 3$. Now, if we let
\[
    L_t = \{ 0, t-2, t-1, t, t+1, 2t-1 \}
\]
be the set consisting of the powers of these six terms, then it can be shown that this set has one valuable property, as described in the following lemma.

\begin{lemma}\label{unique_remainder}
    For each odd $t \ge 3$ and each prime number $p \ge 5$, the set $L_t$ contains an element whose remainder modulo $p$ is unique within the set.
\end{lemma}
\begin{proof}
    The four integers $t-2, t-1, t, t+1$ are consecutive, hence it is clear that they all have mutually distinct remainders modulo $p$. Regardless of what their remainders modulo $p$ are, at least two of these integers must have a remainder that is different from the remainders of both $0$ and $2t-1$.
\end{proof}

Lemma \ref{unique_remainder} proves to be useful while demonstrating one key property of the $Q_t(x)$ and $R_t(x)$ polynomials regarding their division by cyclotomic polynomials. This is shown in the next lemma.

\begin{lemma}\label{square_free}
    For any odd $t \ge 3$, neither $Q_t(x)$ nor $R_t(x)$ can be divisible by a cyclotomic polynomial $\Phi_b(x)$ such that $p^2 \mid b$ for some prime number $p$.
\end{lemma}
\begin{proof}
    Let $b \in \mathbb{N}$ be such that $p^2 \mid b$ for some prime number $p$. In this case, $\frac{b}{p}$ is a positive integer divisible by $p$, hence we get that $\Phi_b(x) = \Phi_{\frac{b}{p}}(x^p)$ (see, for example, \cite[p.\ 160]{Nagell}).

    Now, suppose that $\Phi_b(x) \mid Q_t(x)$ for some odd integer $t \ge 3$. It follows that $Q_t(x) = V(x) \Phi_b(x)$ for some polynomial $V(x) \in \mathbb{Q}[x]$. Let $Q_t^{(j)}(x)$ be the polynomial composed of all the terms of $Q_t(x)$ whose powers are congruent to $j$ modulo $p$, for each $j = \overline{0, p-1}$. Similarly, let $V^{(j)}(x)$ be the polynomial composed of all the terms of $V(x)$ whose powers are congruent to $j$ modulo $p$, for each $j = \overline{0, p-1}$. Given the fact that the powers of all the terms of $\Phi_b(x)$ are divisible by $p$, it swiftly follows that
    \[
        Q_t^{(j)}(x) = V^{(j)}(x) \Phi_b(x)
    \]
    must hold for all the $j = \overline{0, p-1}$. In other words, we necessarily have
    \[
        \Phi_b(x) \mid Q_t^{(j)}(x)
    \]
    for all the $j = \overline{0, p-1}$. The same reasoning and notation can be applied in the case that we suppose that $\Phi_b(x) \mid R_t(x)$ is true for some odd $t \ge 3$. We will finalize the proof of the lemma by taking into consideration three separate cases depending on the value of the prime number $p$.

    \bigskip\noindent
    \emph{Case $p \ge 5$}.\quad
    In this case, Lemma \ref{unique_remainder} dictates that the set $L_t$ must contain an element whose remainder modulo $p$ is unique within that set. In other words, there exists a $j,\, 0 \le j \le p-1$ such that both $Q_t^{(j)}(x)$ and $R_t^{(j)}(x)$ consist of exactly one non-zero term, i.e.\ have the form $c \, x^a$ for some $c \in \mathbb{Z} \setminus \{ 0 \}$ and $a \in \mathbb{N}_0$. If we suppose that either $\Phi_b(x) \mid Q_t(x)$ or $\Phi_b(x) \mid R_t(x)$, this implies that $\Phi_b(x)$ must necessarily divide a polynomial with the aforementioned form $c \, x^a$, which is clearly impossible.

    \bigskip\noindent
    \emph{Case $p = 3$}.\quad
    This case can be quickly resolved by taking into consideration the modular values given in Table \ref{the_remainders_3}. We divide the case into three subcases depending on the value of $t \bmod 3$.

    \begin{table}[h!t]
    {\footnotesize
    \begin{center}
    \begin{tabular}{rccc}
    \toprule & $t \equiv_3 0$ & $t \equiv_3 1$ & $t \equiv_3 2$ \\
    \midrule
    $0 \bmod 3$ & $0$ & $0$ & $0$\\
    $(t-2) \bmod 3$ & $1$ & $2$ & $0$\\
    $(t-1) \bmod 3$ & $2$ & $0$ & $1$\\
    $t \bmod 3$ & $0$ & $1$ & $2$\\
    $(t+1) \bmod 3$ & $1$ & $2$ & $0$\\
    $(2t-1) \bmod 3$ & $2$ & $1$ & $0$\\
    \bottomrule
    \end{tabular}
    \end{center}
    \caption{The elements of the set $L_t$ modulo $3$.}
    \label{the_remainders_3}
    }
    \end{table}

    \medskip\noindent
    \emph{Subcase $t \equiv_3 0$}.\quad
    If we suppose that $\Phi_b(x) \mid Q_t(x)$, we get $\Phi_b(x) \mid x^{t+1} - x^{t-2}$. Similarly, $\Phi_b(x) \mid R_t(x)$ would imply $\Phi_b(x) \mid -x^{t+1} + x^{t-2}$. Either way, we obtain
    \begin{alignat*}{2}
        && \Phi_b(x) &\mid x^{t+1} - x^{t-2}\\
        \implies \quad && \Phi_b(x) &\mid x^{t-2} (x^3 - 1)\\
        \implies \quad && \Phi_b(x) &\mid x^3 - 1 ,
    \end{alignat*}
    which is not possible.

    \medskip\noindent
    \emph{Subcase $t \equiv_3 1$}.\quad
    Regardless of whether we suppose that $\Phi_b(x) \mid Q_t(x)$ holds or $\Phi_b(x) \mid R_t(x)$, we obtain that $\Phi_b(x) \mid x^{t+1} - x^{t-2}$ must be true, which is impossible as we have already demonstrated in the previous subcase.

    \medskip\noindent
    \emph{Subcase $t \equiv_3 2$}.\quad
    Bearing in mind the modular values given in Table \ref{the_remainders_3}, we conclude that $\Phi_b(x) \mid Q_t(x)$ or $\Phi_b(x) \mid R_t(x)$ would surely imply that $\Phi_b(x) \mid y^t$, which is not possible.
    
    \bigskip\noindent
    \emph{Case $p = 2$}.\quad
    In this case, we have that the integers $t-2, t, 2t-1$ are odd, while $0, t-1, t+1$ are even. We will resolve this case separately for $Q_t(x)$ and $R_t(x)$, thus yielding two subcases.

    \medskip\noindent
    \emph{Subcase $\Phi_b(x) \mid Q_t(x)$}.\quad
    Suppose that $\Phi_b(x) \mid Q_t(x)$. We now get
    \begin{align}
        \label{aux_1}\Phi_b(x) &\mid 2 x^{2t-1} - x^t - x^{t-2} ,\\
        \label{aux_2}\Phi_b(x) &\mid x^{t+1} + x^{t-1} - 2 .
    \end{align}
    If we take into consideration that $2 x^{2t-1} - x^t - x^{t-2} = x^{t-2} (2x^{t+1} - x^2 - 1)$, Eq.~(\ref{aux_1}) helps us obtain
    \begin{equation}\label{aux_3}
        \Phi_b(x) \mid 2x^{t+1} - x^2 - 1 .
    \end{equation}
    By subtracting the right-hand sides of Eqs.\ (\ref{aux_2}) and (\ref{aux_3}), we further conclude that
    \begin{alignat}{2}
        \nonumber && \Phi_b(x) &\mid x^{t+1} - x^{t-1} - x^2 + 1\\
        \nonumber \implies \quad && \Phi_b(x) &\mid (x^{t-1} - 1)(x^2 - 1)\\
        \label{aux_4} \implies \quad && \Phi_b(x) &\mid x^{t-1} - 1 .
    \end{alignat}
    Now, if we subtract the right-hand sides of Eqs.\ (\ref{aux_2}) and (\ref{aux_4}), we get
    \[
        \Phi_b(x) \mid x^{t+1} - 1 ,
    \]
    which directly implies
    \begin{alignat*}{2}
        && \Phi_b(x) &\mid (x^{t+1} - 1) - (x^{t-1} - 1)\\
        \implies \quad && \Phi_b(x) &\mid x^{t+1} - x^{t-1}\\
        \implies \quad && \Phi_b(x) &\mid x^{t-1}(x^2 - 1)\\
        \implies \quad && \Phi_b(x) &\mid x^2 - 1 ,
    \end{alignat*}
    which is not possible. Hence, $\Phi_b(x) \mid Q_t(x)$ cannot hold, as desired.

    \medskip\noindent
    \emph{Subcase $\Phi_b(x) \mid R_t(x)$}.\quad
    If we suppose that $\Phi_b(x) \mid R_t(x)$, it follows that
    \begin{align}
        \label{aux_5}\Phi_b(x) &\mid 2 x^{2t-1} - 3x^t + x^{t-2} ,\\
        \label{aux_6}\Phi_b(x) &\mid -x^{t+1} + 3x^{t-1} - 2 .
    \end{align}
    Given the fact that $2 x^{2t-1} - 3x^t + x^{t-2} = x^{t-2} (2x^{t+1} - 3x^2 + 1)$, Eq.~(\ref{aux_5}) yields
    \begin{equation}\label{aux_7}
        \Phi_b(x) \mid 2x^{t+1} - 3x^2 + 1 .
    \end{equation}
    Now, by subtracting the right-hand sides of Eqs.\ (\ref{aux_6}) and (\ref{aux_7}), we further obtain
    \begin{alignat}{2}
        \nonumber && \Phi_b(x) &\mid 3 x^{t+1} - 3x^{t-1} - 3x^2 + 3\\
        \nonumber \implies \quad && \Phi_b(x) &\mid 3(x^{t-1} - 1)(x^2 - 1)\\
        \label{aux_8} \implies \quad && \Phi_b(x) &\mid x^{t-1} - 1 .
    \end{alignat}
    If we take the right-hand side of Eq.\ (\ref{aux_8}), multiply it by two, then subtract it from the right-hand side of Eq.\ (\ref{aux_6}), we will get another polynomial that is divisible by $\Phi_b(x)$. In other words, we have
    \begin{alignat*}{2}
        && \Phi_b(x) &\mid (-x^{t+1} + 3x^{t-1} - 2) - 2(x^{t-1} - 1)\\
        \implies \quad && \Phi_b(x) &\mid -x^{t+1} + x^{t-1}\\
        \implies \quad && \Phi_b(x) &\mid x^{t-1}(1 - x^2)\\
        \implies \quad && \Phi_b(x) &\mid 1 - x^2 ,
    \end{alignat*}
    which is not possible. Thus, $\Phi_b(x) \mid R_t(x)$ cannot be true.
\end{proof}

Lemma \ref{square_free} tells us that the only cyclotomic polynomials $\Phi_b(x)$ that could divide either $Q_t(x)$ or $R_t(x)$ are those where $b \in \mathbb{N}$ is a square-free integer. In fact, for any odd $t \ge 3$, the only cyclotomic polynomials that divide these polynomials are $\Phi_1(b)$ and $\Phi_2(b)$. This observation will prove to be of the utmost importance while formulating the complete proof of Theorem \ref{main_theorem_1} later on. However, in order to prove this statement, we shall need two more additional lemmas. We begin with the next one.

\begin{lemma}\label{p_2p}
    For each odd $t \ge 3$ and each prime number $p \ge 7$, neither $Q_t(x)$ nor $R_t(x)$ can be divisible by the cyclotomic polynomial $\Phi_p(x)$ or the cyclotomic polynomial $\Phi_{2p}(x)$.
\end{lemma}
\begin{proof}
    First of all, it is important to notice that for all prime numbers $p \ge 7$ we have
    \begin{align*}
        \Phi_p(x) &= \sum_{j = 0}^{p-1} x^j,\\
        \Phi_{2p}(x) &= \sum_{j = 0}^{p-1} (-1)^j x^j .
    \end{align*}
    Here, it is clear that $\deg \Phi_p = \deg \Phi_{2p} = p-1$. We will finish the proof of the lemma by splitting the problem to two separate cases depending on whether we are dealing with $\Phi_p(x)$ or $\Phi_{2p}(x)$.

    \bigskip\noindent
    \emph{Case $\Phi_p(x)$}.\quad
    Let $Q_t^{\bmod p}(x)$ and $R_t^{\bmod p}(x)$ be the following two polynomials:
    \begin{align*}
        Q_t^{\bmod p}(x) &= 2x^{(2t-1) \bmod p} + x^{(t+1) \bmod p} - x^{t \bmod p} + x^{(t-1) \bmod p} - x^{(t-2) \bmod p} - 2,\\
        R_t^{\bmod p}(x) &= 2x^{(2t-1) \bmod p} - x^{(t+1) \bmod p} - 3x^{t \bmod p} + 3x^{(t-1) \bmod p} + x^{(t-2) \bmod p} - 2.
    \end{align*}
    Suppose that $\Phi_p(x) \mid Q_t(x)$. In this case, it is clear that $\Phi_p(x) \mid Q_t^{\bmod p}(x)$ must hold, too. Given the fact that $\deg Q_t^{\bmod p} \le p-1 = \deg \Phi_p$, we further get two possibilities:
    \begin{itemize}
        \item $Q_t^{\bmod p}(x) \equiv 0$;
        \item $Q_t^{\bmod p}(x) = c \, \Phi_p(x)$ for some $c \in \mathbb{Q} \setminus \{ 0 \}$.
    \end{itemize}
    By virtue of Lemma \ref{unique_remainder}, there must exist a non-zero term in $Q_t(x)$ whose power has a unique remainder modulo $p$. This directly implies that $Q_t^{\bmod p}(x) \equiv 0$ cannot hold. On the other hand, $Q_t^{\bmod p}(x) = c \, \Phi_p(x)$ implies that $Q_t^{\bmod p}(x)$ and $\Phi_p(x)$ must have the same number of non-zero terms, i.e.\ $Q_t^{\bmod p}(x)$ needs to have exactly $p$ non-zero terms. This is clearly impossible due to the fact that $Q_t^{\bmod p}(x)$ has at most six non-zero terms. We conclude that $\Phi_p(x) \mid Q_t(x)$ cannot be true.

    The proof of $\Phi_p(x) \nmid R_t(x)$ is completely analogous to the previously described proof of $\Phi_p(x) \nmid Q_t(x)$. Thus, it will be left out.
    
    \bigskip\noindent
    \emph{Case $\Phi_{2p}(x)$}.\quad
    Let $\hat{Q}_t^{\bmod p}(x)$ and $\hat{R}_t^{\bmod p}(x)$ be the following two polynomials:
    \begin{align*}
        \hat{Q}_t^{\bmod p}(x) = 2&(-1)^{\lfloor\frac{2t-1}{p}\rfloor}x^{(2t-1) \bmod p} + (-1)^{\lfloor\frac{t+1}{p}\rfloor}x^{(t+1) \bmod p} - (-1)^{\lfloor\frac{t}{p}\rfloor}x^{t \bmod p}\\
        & \qquad + (-1)^{\lfloor\frac{t-1}{p}\rfloor}x^{(t-1) \bmod p} - (-1)^{\lfloor\frac{t-2}{p}\rfloor}x^{(t-2) \bmod p} - 2 ,\\
        \hat{R}_t^{\bmod p}(x) = 2&(-1)^{\lfloor\frac{2t-1}{p}\rfloor}x^{(2t-1) \bmod p} - (-1)^{\lfloor\frac{t+1}{p}\rfloor}x^{(t+1) \bmod p} - 3(-1)^{\lfloor\frac{t}{p}\rfloor}x^{t \bmod p}\\
        & \qquad + 3(-1)^{\lfloor\frac{t-1}{p}\rfloor}x^{(t-1) \bmod p} + (-1)^{\lfloor\frac{t-2}{p}\rfloor}x^{(t-2) \bmod p} - 2 .
    \end{align*}
    Since each primitive $2p$-th root of unity gives $-1$ when raised to the power of $p$, it becomes evident that $\Phi_{2p}(x) \mid Q_t(x)$ is equivalent to $\Phi_{2p}(x) \mid \hat{Q}_t^{\bmod p}(x)$. Likewise, $\Phi_{2p}(x) \mid R_t(x)$ will hold if and only if $\Phi_{2p}(x) \mid \hat{R}_t^{\bmod p}(x)$ does, too.

    Now, suppose that $\Phi_{2p}(x) \mid Q_t(x)$. In this case we obtain $\Phi_{2p} \mid \hat{Q}_t^{\bmod p}(x)$, where $\deg \hat{Q}_t^{\bmod p} \le p-1 = \deg \Phi_{2p}$. Like in the previous case, we get two possible options:
    \begin{itemize}
        \item $\hat{Q}_t^{\bmod p}(x) \equiv 0$;
        \item $\hat{Q}_t^{\bmod p}(x) = c \, \Phi_{2p}(x)$ for some $c \in \mathbb{Q} \setminus \{ 0 \}$.
    \end{itemize}
    Lemma \ref{unique_remainder} dictates that there must exist a non-zero term in $Q_t(x)$ whose power has a unique remainder modulo $p$, which means that it is impossible for $\hat{Q}_t^{\bmod p}(x) \equiv 0$ to be true. On the other hand, from $\hat{Q}_t^{\bmod p}(x) = c \, \Phi_{2p}(x)$ we get that $\hat{Q}_t^{\bmod p}(x)$ must have exactly $p$ non-zero terms. However, this is not possible since it is clear that this polynomial cannot have more than six non-zero terms. It swiftly follows that $\Phi_{2p}(x) \nmid Q_t(x)$.

    The proof of $\Phi_{2p}(x) \nmid R_t(x)$ is entirely analogous to the elaborated proof of $\Phi_{2p}(x) \nmid Q_t(x)$. For this reason, we choose to leave it out.
\end{proof}

As the final piece of the puzzle, we will need to show that some concrete cyclotomic polynomials cannot divide $Q_t(x)$ or $R_t(x)$, for each odd integer $t \ge 3$. More precisely, we will turn our interest to the cyclotomic polynomials $\Phi_b(x)$ where $b$ is a square-free integer all of whose prime factors belong to the set $\{2, 3, 5\}$.

\begin{lemma}\label{small_ones}
    For each odd $t \ge 3$ and each positive integer $b \in \{ 3, 5, 6, 10, 15, 30 \}$, the cyclotomic polynomial $\Phi_b(x)$ divides neither $Q_t(x)$ nor $R_t(x)$.
\end{lemma}
\begin{proof}
    Similarly as in the proof of Lemma \ref{p_2p}, let $Q_t^{\bmod b}(x)$ and $R_t^{\bmod b}(x)$ be the next two polynomials:
    \begin{align*}
        Q_t^{\bmod b}(x) &= 2x^{(2t-1) \bmod b} + x^{(t+1) \bmod b} - x^{t \bmod b} + x^{(t-1) \bmod b} - x^{(t-2) \bmod b} - 2,\\
        R_t^{\bmod b}(x) &= 2x^{(2t-1) \bmod b} - x^{(t+1) \bmod b} - 3x^{t \bmod b} + 3x^{(t-1) \bmod b} + x^{(t-2) \bmod b} - 2.
    \end{align*}
    Here, it is imperative for us to notice that $\Phi_b(x) \mid Q_t(x) \iff \Phi_b(x) \mid Q_t^{\bmod b}(x)$ and also that $\Phi_b(x) \mid R_t(x) \iff \Phi_b(x) \mid R_t^{\bmod b}(x)$. However, for a fixed value of $b \in \mathbb{N}$, there exist only finitely many polynomials $Q_t^{\bmod b}(x)$ and $R_t^{\bmod b}(x)$, as $t$ ranges over the odd integers greater than or equal to three. This is a direct consequence of the fact that $t$ can only have finitely many remainders modulo $b$. Hence, if we are able to show that $\Phi_b(x)$ divides none of these concrete polynomials, this is sufficient to prove that $\Phi_b(x)$ divides neither $Q_t(x)$ nor $R_t(x)$.

    In order to prove the lemma, it is enough to demonstrate that $\Phi_b(x) \nmid Q_t^{\bmod b}(x)$ and $\Phi_b(x) \nmid R_t^{\bmod b}(x)$ for all the $b \in \{3, 5, 6, 10, 15, 30\}$ and for all the possible remainders $t \bmod b$. Doing this is trivial via computer with the help of some symbolic computation software. We disclose the required computational results in the form of two tables which are given in Appendices \ref{appendix_qt} and \ref{appendix_rt}. These results clearly indicate that for each $b \in \{3, 5, 6, 10, 15, 30\}$ the remainders $Q_t^{\bmod b}(x) \bmod \Phi_b(x)$ and $R_t^{\bmod b}(x) \bmod \Phi_b(x)$ can never be equal to the zero polynomial, regardless of what the value of $t \bmod b$ is, which completes the proof of the lemma.
\end{proof}

It is important to notice that Theorem \ref{filaseta} can very conveniently be used on the polynomials $Q_t(x)$ and $R_t(x)$, given the fact that they only have six non-zero terms. This means that if $\Phi_b(x) \mid Q_t(x)$ and $p \mid b$ for some square-free integer $b$ and some prime number $p \ge 7$, we can immediately deduct that $\Phi_{\frac{b}{p}}(x) \mid Q_t(x)$ is true as well. The same can be said regarding $R_t(x)$. This practically means that if the aforementioned division holds, we can cancel out as many prime factors of $b$ that are not below seven as we want, and the division will still have to hold. Bearing this in mind, we are now able to formulate and prove the following lemma regarding the divisibility of $Q_t(x)$ and $R_t(x)$ by cyclotomic polynomials.

\begin{lemma}\label{main_cyclotomic}
    For each odd $t \ge 3$, neither $Q_t(x)$ nor $R_t(x)$ can be divisible by a cyclotomic polynomial $\Phi_b(x)$ where $b \ge 3$.
\end{lemma}
\begin{proof}
    The proof will only be given for $Q_t(x)$, given the fact that the proof regarding $R_t(x)$ is completely analogous. Suppose that $\Phi_b(x) \mid Q_t(x)$ for some $b \ge 3$. By virtue of Lemma \ref{square_free}, we know that $b$ needs to be a square-free integer. We now divide the problem into two separate cases, depending on whether $b$ is divisible by either $3$ or $5$, or not.

    \bigskip\noindent
    \emph{Case $3 \nmid b$ and $5 \nmid b$}.\quad
    In this case, it is clear that $b$ has at least one prime factor greater than $5$, since $b \notin \{1, 2\}$. If we repeatedly use Theorem \ref{filaseta} in order to cancel out all the prime factors of $b$ that are greater than $5$, until exactly one is left, we obtain that
    \[
        \Phi_{b'}(x) \mid Q_t(x)
    \]
    must hold, where $b'$ is square-free, not divisible by $3$ or $5$, and has exactly one prime factor greater than $5$. In other words, we get that either $\Phi_p(x) \mid Q_t(x)$ or $\Phi_{2p} \mid Q_t(x)$ holds, for some prime number $p \ge 7$. However, this is not possible according to Lemma \ref{p_2p}.

    \bigskip\noindent
    \emph{Case $3 \mid b$ or $5 \mid b$}.\quad
    Here, we can simply use Theorem \ref{filaseta} in order to cancel out all the prime factors of $b$ that are greater than $5$, until there are none left. This leads us to
    \[
        \Phi_{b'}(x) \mid Q_t(x),
    \]
    where $b'$ is square-free, divisible by $3$ or $5$, and has no prime factors outside of the set $\{2, 3, 5\}$. These conditions imply that $b' \in \{3, 5, 6, 10, 15, 30\}$. However, in this case $\Phi_{b'}(x) \mid Q_t(x)$ cannot possibly hold, as a direct consequence of Lemma~\ref{small_ones}.

    \bigskip
    Both of the cases have led to a contradiction, which means that $\Phi_b(x) \mid Q_t(x)$ cannot hold for any $b \ge 3$.
\end{proof}

Lemma \ref{main_cyclotomic} allows us to finish the proof of Theorem \ref{main_theorem_1}. We present the rest of the proof in the remainder of this section.

\bigskip\noindent
\emph{Proof of Theorem \ref{main_theorem_1}}.\quad
The proof for the case $t = 1$ is fairly straightforward. Due to the fact that $\mathcal{S}'_{1, n} = \left\{ \frac{n}{4}, \frac{n}{4} + 1 \right\}$, Eq.\ (\ref{polynomial_formula}) immediately gives us
\begin{equation}\label{case_t1_1}
    P(\zeta) = \zeta^{\frac{n}{4}} + \dfrac{1}{\zeta^{\frac{n}{4}}} + \zeta^{\frac{n}{4}+1} + \dfrac{1}{\zeta^{\frac{n}{4}+1}} ,
\end{equation}
where $\zeta$ is an arbitrary $n$-th root of unity different from $1$ and $-1$. Now, the condition $P(\zeta) = 0$ becomes equivalent to
\begin{alignat*}{2}
    && P(\zeta) &= 0\\
    \iff \quad && \zeta^{\frac{n}{4}+1} \left( \zeta^{\frac{n}{4}} + \dfrac{1}{\zeta^{\frac{n}{4}}} + \zeta^{\frac{n}{4}+1} + \dfrac{1}{\zeta^{\frac{n}{4}+1}} \right) &=0\\
    \iff \quad && \zeta^{\frac{n}{2} + 2} + \zeta^{\frac{n}{2} + 1} + \zeta + 1 &= 0\\
    \iff \quad && (\zeta + 1)\left(\zeta^{\frac{n}{2} + 1} + 1 \right) &= 0\\
    \iff \quad && \zeta^{\frac{n}{2} + 1} + 1 &= 0 .
\end{alignat*}
Since we know that $\zeta^{\frac{n}{2}} \in \{1, -1\}$, we obtain that $\zeta^{\frac{n}{2} + 1} + 1$ is equal to $\zeta + 1$ or $-\zeta + 1$. Either way, this value cannot be equal to $0$, hence we conclude that $P(\zeta) \neq 0$. Since the value $\zeta$ was arbitrarily chosen, we obtain that no $n$-th root of unity different from $1$ and $-1$ can be a root of $P(x)$. According to Lemma \ref{damnjanovic_lemma_1}, this means that $\mathcal{D}'_{1, n}$ is indeed a circulant nut graph for any $n \ge 8$ divisble by four, as desired.

Now, we turn our attention to the case when $t \ge 3$. In this scenario, Eq.\ (\ref{polynomial_formula}) gives us
\[
    P(\zeta) = \left( \zeta^{\frac{n}{4}} + \frac{1}{\zeta^{\frac{n}{4}}} \right) + \left( \zeta^{\frac{n}{4}+1} + \frac{1}{\zeta^{\frac{n}{4}+1}} \right) + \sum_{j=1}^{t-1}\left( \zeta^j + \frac{1}{\zeta^j}\right) + \sum_{j=\frac{n}{2}-t+1}^{\frac{n}{2}-1}\left( \zeta^j + \frac{1}{\zeta^j}\right),
\]
where $\zeta$ is an arbitrarily chosen $n$-th root of unity different from $1$ and $-1$. It is easy to further conclude that
\begin{equation}\label{polynomial_formula_2}
    P(\zeta) = \left( \zeta^{\frac{n}{4}+1} + \zeta^{\frac{n}{4}} + \frac{1}{\zeta^{\frac{n}{4}}} + \frac{1}{\zeta^{\frac{n}{4}+1}} \right) + \sum_{j=1}^{t-1}\left( \zeta^j + \frac{1}{\zeta^j} + \zeta^{\frac{n}{2} - j} + \frac{1}{\zeta^{\frac{n}{2}-j}} \right) .
\end{equation}
We will finish the proof by showing that $P(\zeta) \neq 0$ must necessarily hold. In order to make the proof more concise, we will divide it into two cases depending on whether $\zeta^{\frac{n}{2}}$ is equal to $1$ or $-1$.

\bigskip\noindent
\emph{Case $\zeta^{\frac{n}{2}} = -1$}.\quad
In this case, we have that $\zeta^{\frac{n}{2} - j} = -\dfrac{1}{\zeta^j}$ and $\dfrac{1}{\zeta^{\frac{n}{2} - j}} = -\zeta^j$, which quickly implies
\[
    \zeta^j + \frac{1}{\zeta^j} + \zeta^{\frac{n}{2} - j} + \frac{1}{\zeta^{\frac{n}{2}-j}} = 0
\]
for any $j = \overline{1, t-1}$. Thus, Eq.\ (\ref{polynomial_formula_2}) simplifies to the following formula:
\[
    P(\zeta) = \zeta^{\frac{n}{4}+1} + \zeta^{\frac{n}{4}} + \frac{1}{\zeta^{\frac{n}{4}}} + \frac{1}{\zeta^{\frac{n}{4}+1}} .
\]
However, this formula is exactly the same as Eq.\ (\ref{case_t1_1}), which we got while dealing with the case $t = 1$. Hence, an almost absolutely identical proof can be used in order to show that $P(\zeta) \neq 0$.

\bigskip\noindent
\emph{Case $\zeta^{\frac{n}{2}} = 1$}.\quad
Here, we obtain that $\zeta^{\frac{n}{2} - j} = \dfrac{1}{\zeta^j}$ and $\dfrac{1}{\zeta^{\frac{n}{2} - j}} = \zeta^j$. This means that
\[
    \zeta^j + \frac{1}{\zeta^j} + \zeta^{\frac{n}{2} - j} + \frac{1}{\zeta^{\frac{n}{2}-j}} = 2 \left( \zeta^j + \frac{1}{\zeta^j} \right)
\]
for any $j = \overline{1, t-1}$. According to Eq.\ (\ref{polynomial_formula_2}), the condition $P(\zeta) = 0$ becomes equivalent to
\begin{alignat}{2}
    \nonumber && P(\zeta) &= 0\\
    \nonumber \iff \quad && \left( \zeta^{\frac{n}{4}+1} + \zeta^{\frac{n}{4}} + \frac{1}{\zeta^{\frac{n}{4}}} + \frac{1}{\zeta^{\frac{n}{4}+1}} \right) + 2 \sum_{j=1}^{t-1}\left( \zeta^j + \frac{1}{\zeta^j} \right) &= 0\\
    \nonumber \iff \quad && \left( \zeta^{\frac{n}{4}+1} + \zeta^{\frac{n}{4}} + \frac{1}{\zeta^{\frac{n}{4}}} + \frac{1}{\zeta^{\frac{n}{4}+1}} \right) - 2 + 2 \sum_{j=1-t}^{t-1} \zeta^j &= 0\\
    \nonumber \iff \quad && \zeta^{t-1} \left( \left( \zeta^{\frac{n}{4}+1} + \zeta^{\frac{n}{4}} + \frac{1}{\zeta^{\frac{n}{4}}} + \frac{1}{\zeta^{\frac{n}{4}+1}} \right) - 2 + 2 \sum_{j=1-t}^{t-1} \zeta^j \right) &= 0\\
    \nonumber \iff \quad && \left( \zeta^{t+\frac{n}{4}} + \zeta^{t+\frac{n}{4}-1} + \zeta^{t-\frac{n}{4}-1} + \zeta^{t-\frac{n}{4}-2} \right) - 2\zeta^{t-1} + 2 \sum_{j=0}^{2t-2} \zeta^j &= 0\\
    \nonumber \iff \quad && (\zeta - 1)\left( \zeta^{t+\frac{n}{4}} + \zeta^{t+\frac{n}{4}-1} + \zeta^{t-\frac{n}{4}-1} + \zeta^{t-\frac{n}{4}-2} - 2\zeta^{t-1} + 2 \sum_{j=0}^{2t-2} \zeta^j \right) &= 0\\
    \label{polynomial_formula_3} \iff \quad && \zeta^{t+\frac{n}{4}+1} - \zeta^{t+\frac{n}{4}-1} + \zeta^{t-\frac{n}{4}} - \zeta^{t-\frac{n}{4}-2} - 2\zeta^t + 2\zeta^{t-1} + 2 \zeta^{2t-1} - 2 &= 0 .
\end{alignat}

Now, we will divide the problem into two subcases depending on whether $\zeta^{\frac{n}{4}}$ is equal to $1$ or $-1$.

\medskip\noindent
\emph{Subcase $\zeta^{\frac{n}{4}} = 1$}.\quad
In this subcase, it is easy to see that Eq.\ (\ref{polynomial_formula_3}) further becomes equivalent to
\begin{alignat*}{2}
    && \zeta^{t+\frac{n}{4}+1} - \zeta^{t+\frac{n}{4}-1} + \zeta^{t-\frac{n}{4}} - \zeta^{t-\frac{n}{4}-2} - 2\zeta^t + 2\zeta^{t-1} + 2 \zeta^{2t-1} - 2 &= 0\\
    \iff \quad && \zeta^{t+1} - \zeta^{t-1} + \zeta^{t} - \zeta^{t-2} - 2\zeta^t + 2\zeta^{t-1} + 2 \zeta^{2t-1} - 2 &= 0\\
    \iff \quad && 2\zeta^{2t-1} + \zeta^{t+1} - \zeta^t  + \zeta^{t-1} - \zeta^{t-2} - 2 &= 0 .
\end{alignat*}
Thus, if we suppose that $P(\zeta) = 0$, we then get that $\zeta$ is a root of the polynomial $Q_t(x)$. However, since $\zeta \neq 1, -1$ and $\zeta^n = 1$, this means that $\zeta$ is a primitive $b$-th root of unity for some $b \ge 3$. This implies that $\Phi_b(x)$ necessarily divides $Q_t(x)$, which is impossible according to Lemma \ref{main_cyclotomic}, hence $P(\zeta) = 0$ cannot be true.

\medskip\noindent
\emph{Subcase $\zeta^{\frac{n}{4}} = -1$}.\quad
Here, Eq.\ (\ref{polynomial_formula_3}) quickly becomes equivalent to
\begin{alignat*}{2}
    && \zeta^{t+\frac{n}{4}+1} - \zeta^{t+\frac{n}{4}-1} + \zeta^{t-\frac{n}{4}} - \zeta^{t-\frac{n}{4}-2} - 2\zeta^t + 2\zeta^{t-1} + 2 \zeta^{2t-1} - 2 &= 0\\
    \iff \quad && -\zeta^{t+1} + \zeta^{t-1} - \zeta^{t} + \zeta^{t-2} - 2\zeta^t + 2\zeta^{t-1} + 2 \zeta^{2t-1} - 2 &= 0\\
    \iff \quad && 2\zeta^{2t-1} - \zeta^{t+1} - 3\zeta^t  + 3\zeta^{t-1} + \zeta^{t-2} - 2 &= 0 .
\end{alignat*}
In a similar fashion, if we suppose that $P(\zeta) = 0$, we conclude that $\zeta$ must be a root of the polynomial $R_t(x)$. Due to the fact that $\zeta \neq 1, -1$ and $\zeta^n = 1$, we see that $\zeta$ is a primitive $b$-th root of unity for some $b \ge 3$. Hence, $\Phi_b(x) \mid R_t(x)$, which is again impossible by virtue of Lemma \ref{main_cyclotomic}. Thus, $P(\zeta) \neq 0$.
\hfill\qed

\section{\texorpdfstring{$\mathcal{D}''_{t, n}$}{D"(t, n)} family of circulant graphs}\label{d2_family}

Here, we give the second of the two central theorems disclosed in the paper.
\begin{theorem}\label{main_theorem_2}
    For each $t \in \mathbb{N}$ and $n \ge 4t + 6$ such that $n \equiv_4 2$, the circulant graph $\mathcal{D}''_{t, n}$ must be a $4t$-regular nut graph of order $n$.
\end{theorem}

The complete proof of Theorem \ref{main_theorem_2} will have a very similar structure as the previously described proof of Theorem \ref{main_theorem_1}, albeit with more complexity. Instead of the polynomials $Q_t(x)$ and $R_t(x)$, here we will rely on the integer polynomials $U_t, W_t \in \mathbb{Z}[x]$ that are defined as
\begin{align*}
    U_t(x) &= 2x^{4t-1} + x^{2t+4} - 2x^{2t+1} + 2x^{2t-1} - x^{2t-4} - 2x,\\
    W_t(x) &= 2x^{4t-1} - x^{2t+4} - 2x^{2t+1} + 2x^{2t-1} + x^{2t-4} - 2x,
\end{align*}
for all the $t \in \mathbb{N},\, t \ge 2$. These polynomials are clearly well defined for all the $t \ge 2$ and must have exactly six non-zero terms, due to the fact that
\[
    1 < 2t - 4 < 2t - 1 < 2t + 1 < 2t + 4 < 4t - 1
\]
for each $t \ge 3$, while
\begin{align*}
    4t - 1 &= 7, & 2t + 4 &= 8, & 2t + 1 &= 5,\\
    2t - 1 &= 3, & 2t - 4 &= 0, & 1 &= 1,
\end{align*}
for $t = 2$. By setting
\[
    M_t = \{1, 2t-4, 2t-1, 2t+1, 2t+4, 4t-1 \}
\]
for each $t \ge 2$, we can show that the $M_t$ sets have a property that is very similar to the one regarding the $L_t$ sets which was demonstrated earlier in Lemma \ref{unique_remainder}.

\begin{lemma}\label{unique_remainders_2}
    For each $t \ge 2$ and each prime number $p \ge 7$, the set $M_t$ contains an element whose remainder modulo $p$ is unique within the set.
\end{lemma}
\begin{proof}
    It is easy to notice that the four integers $2t-4, 2t-1, 2t+1, 2t+4$ necessarily have mutually distinct remainders modulo $p$. This quickly implies that at least two of them must have a remainder modulo $p$ that is different from the remainders of both $1$ and $4t-1$.
\end{proof}

We will now implement Lemma \ref{unique_remainders_2} in order to prove a valuable property of the $U_t(x)$ and $W_t(x)$ polynomials regarding their division by cyclotomic polynomials, as demonstrated in the following lemma.

\begin{lemma}\label{square_free_2}
    For any $t \ge 2$, neither $U_t(x)$ nor $W_t(x)$ can be divisible by a cyclotomic polynomial $\Phi_b(x)$ such that $p^2 \mid b$ for some prime number $p \ge 3$.
\end{lemma}
\begin{proof}
    This proof will be done in a completely analogous manner as that of Lemma~\ref{square_free}. In case $b \in \mathbb{N}$ and $p^2 \mid b$ for some prime number $p$, we will again use the fact that $\Phi_b(x) = \Phi_{\frac{b}{p}} (x^p)$.

    Suppose that $\Phi_b(x) \mid U_t(x)$ for some $t \ge 2$ and $b \in \mathbb{N}$ such that $p^2 \mid b$ for some prime number $p \ge 3$. Let $U_t^{(j)}(x)$ denote the polynomial composed of all the terms of $U_t(x)$ whose powers are congruent to $j$ modulo $p$, for each $j = \overline{0, p-1}$. The same logic used in the proof of Lemma \ref{square_free} allows us to conclude that
    \[
        \Phi_b(x) \mid U_t^{(j)}(x)
    \]
    must hold for all the $j = \overline{0, p-1}$. It is clear that the same notation and implication can be used regarding the $W_t(x)$ polynomial in case $\Phi_b(x) \mid W_t(x)$ is true. We will now finalize the proof of the lemma by taking into consideration three separate cases depending on the value of the prime number $p$.

    \bigskip\noindent
    \emph{Case $p \ge 7$}.\quad
    In this scenario, Lemma \ref{unique_remainders_2} tells us that the set $M_t$ must contain an element whose remainder modulo $p$ is unique within that set. Hence, this case can be proved completely analogously to the case $p \ge 5$ from the proof of Lemma \ref{square_free}.

    \bigskip\noindent
    \emph{Case $p = 5$}.\quad
    If $t \equiv_5 2$ or $t \equiv_5 4$, then there exists an element of the set $M_t$ whose remainder modulo $5$ is unique with that set, as shown in Table \ref{remainders_5}. In this subcase, a contradiction can easily be obtained if we suppose that either $\Phi_b(x) \mid U_t(x)$ or $\Phi_b(x) \mid W_t(x)$, by using the same logic as in the $p \ge 7$ case. We resolve the remaining three subcases separately.
    
    \begin{table}[h!t]
    {\footnotesize
    \begin{center}
    \begin{tabular}{rccccc}
    \toprule & $t \equiv_5 0$ & $t \equiv_5 1$ & $t \equiv_5 2$ & $t \equiv_5 3$ & $t \equiv_5 4$ \\
    \midrule
    $1 \bmod 5$ & $1$ & $1$ & $1$ & $1$ & $1$\\
    $(2t-4) \bmod 5$ & $1$ & $3$ & $0$ & $2$ & $4$\\
    $(2t-1) \bmod 5$ & $4$ & $1$ & $3$ & $0$ & $2$\\
    $(2t+1) \bmod 5$ & $1$ & $3$ & $0$ & $2$ & $4$\\
    $(2t+4) \bmod 5$ & $4$ & $1$ & $3$ & $0$ & $2$\\
    $(4t-1) \bmod 5$ & $4$ & $3$ & $2$ & $1$ & $0$\\
    \bottomrule
    \end{tabular}
    \end{center}
    \caption{The elements of the set $M_t$ modulo $5$.}
    \label{remainders_5}
    }
    \end{table}

    \medskip\noindent
    \emph{Subcase $t \equiv_5 0$}.\quad
    If we suppose that $\Phi_b \mid U_t(x)$, then we get
    \begin{align*}
        \Phi_b(x) &\mid -2x^{2t+1} - x^{2t-4} - 2x,\\
        \Phi_b(x) &\mid 2x^{4t-1} + x^{2t+4} + 2x^{2t-1} .
    \end{align*}
    Since $t \ge 5$, we quickly obtain
    \begin{alignat}{2}
        \nonumber && \Phi_b(x) &\mid -2x^{2t+1} - x^{2t-4} - 2x\\
        \nonumber \implies \quad && \Phi_b(x) &\mid -x(2x^{2t} + x^{2t-5} + 2)\\
        \label{aux_9}\implies \quad && \Phi_b(x) &\mid 2x^{2t} + x^{2t-5} + 2 ,
    \end{alignat}
    as well as
    \begin{alignat}{2}
        \nonumber && \Phi_b(x) &\mid 2x^{4t-1} + x^{2t+4} + 2x^{2t-1}\\
        \nonumber \implies \quad && \Phi_b(x) &\mid x^{2t-1}(2x^{2t} + x^5 + 2)\\
        \label{aux_10}\implies \quad && \Phi_b(x) &\mid 2x^{2t} + x^5 + 2 .
    \end{alignat}
    Furthermore, if we subtract the right-hand sides of Eqs.\ (\ref{aux_9}) and (\ref{aux_10}), this leads us to
    \begin{alignat*}{2}
        && \Phi_b(x) &\mid (2x^{2t} + x^{2t-5} + 2) - (2x^{2t} + x^5 + 2)\\
        \implies \quad && \Phi_b(x) &\mid x^{2t-5} - x^5\\
        \implies \quad && \Phi_b(x) &\mid x^5(x^{2t-10} - 1)\\
        \implies \quad && \Phi_b(x) &\mid x^{2t-10} - 1 .
    \end{alignat*}
    Now, Eq.\ (\ref{aux_10}) helps us obtain
    \begin{alignat*}{2}
        && \Phi_b(x) &\mid 2x^{2t} + x^5 + 2\\
        \implies \quad && \Phi_b(x) &\mid 2x^{2t} + x^5 + 2 - 2x^{10}(x^{2t-10} - 1)\\
        \implies \quad && \Phi_b(x) &\mid 2x^{10} + x^5 + 2 .
    \end{alignat*}
    However, none of the roots of the polynomial $2x^{10} + x^5 + 2$ are actually roots of unity, as demonstrated in Appendix \ref{problematic_polynomials}, hence we obtain a contradiction.

    Similarly, if we suppose that $\Phi_b(x) \mid W_t(x)$, then we get
    \begin{align*}
        \Phi_b(x) &\mid -2x^{2t+1} + x^{2t-4} - 2x,\\
        \Phi_b(x) &\mid 2x^{4t-1} - x^{2t+4} + 2x^{2t-1} .
    \end{align*}
    Due to $t \ge 5$, we have
    \begin{alignat}{2}
        \nonumber && \Phi_b(x) &\mid -2x^{2t+1} + x^{2t-4} - 2x\\
        \nonumber \implies \quad && \Phi_b(x) &\mid -x(2x^{2t} - x^{2t-5} + 2)\\
        \label{aux_11}\implies \quad && \Phi_b(x) &\mid 2x^{2t} - x^{2t-5} + 2
    \end{alignat}
    and
    \begin{alignat}{2}
        \nonumber && \Phi_b(x) &\mid 2x^{4t-1} - x^{2t+4} + 2x^{2t-1}\\
        \nonumber \implies \quad && \Phi_b(x) &\mid x^{2t-1}(2x^{2t} - x^5 + 2)\\
        \label{aux_12}\implies \quad && \Phi_b(x) &\mid 2x^{2t} - x^5 + 2 .
    \end{alignat}
    By subtracting the right-hand sides of Eqs.\ (\ref{aux_11}) and (\ref{aux_12}), it follows that
    \begin{alignat*}{2}
        && \Phi_b(x) &\mid (2x^{2t} - x^{2t-5} + 2) - (2x^{2t} - x^5 + 2)\\
        \implies \quad && \Phi_b(x) &\mid - x^{2t-5} + x^5\\
        \implies \quad && \Phi_b(x) &\mid -x^5(x^{2t-10} - 1)\\
        \implies \quad && \Phi_b(x) &\mid x^{2t-10} - 1 .
    \end{alignat*}
    Hence, Eq.\ (\ref{aux_12}) is now able to give us
    \begin{alignat*}{2}
        && \Phi_b(x) &\mid 2x^{2t} - x^5 + 2\\
        \implies \quad && \Phi_b(x) &\mid 2x^{2t} - x^5 + 2 - 2x^{10}(x^{2t-10} - 1)\\
        \implies \quad && \Phi_b(x) &\mid 2x^{10} - x^5 + 2 .
    \end{alignat*}
    However, similarly as with $U_t(x)$, none of the roots of $2x^{10} - x^5 + 2$ are roots of unity, as shown in Appendix \ref{problematic_polynomials}, which leads to a contradiction.

    \medskip\noindent
    \emph{Subcase $t \equiv_5 1$}.\quad
    If we suppose that $\Phi_b \mid U_t(x)$, we immediately get
    \begin{align*}
        \Phi_b(x) &\mid x^{2t+4} + 2x^{2t-1} - 2x,\\
        \Phi_b(x) &\mid 2x^{4t-1} - 2x^{2t+1} - x^{2t-4} .
    \end{align*}
    Since $t \ge 6$, it is easy to see that
    \begin{alignat}{2}
        \nonumber && \Phi_b(x) &\mid x^{2t+4} + 2x^{2t-1} - 2x\\
        \nonumber \implies \quad && \Phi_b(x) &\mid x(x^{2t+3} + 2x^{2t-2} - 2)\\
        \label{aux_13}\implies \quad && \Phi_b(x) &\mid x^{2t+3} + 2x^{2t-2} - 2 ,
    \end{alignat}
    as well as
    \begin{alignat}{2}
        \nonumber && \Phi_b(x) &\mid 2x^{4t-1} - 2x^{2t+1} - x^{2t-4}\\
        \nonumber \implies \quad && \Phi_b(x) &\mid x^{2t-4}(2x^{2t+3} - 2x^5 - 1)\\
        \label{aux_14}\implies \quad && \Phi_b(x) &\mid 2x^{2t+3} - 2x^5 - 1 .
    \end{alignat}
    By using Eqs.\ (\ref{aux_13}) and (\ref{aux_14}) together, we further obtain
    \begin{alignat}{2}
        \nonumber && \Phi_b(x) &\mid 2(x^{2t+3} + 2x^{2t-2} - 2) - (2x^{2t+3} - 2x^5 - 1)\\
        \label{aux_15} \implies \quad && \Phi_b(x) &\mid 4x^{2t-2} + 2x^5 - 3.
    \end{alignat}
    Now, by using Eqs.\ (\ref{aux_14}) and (\ref{aux_15}) together, we see that
    \begin{alignat*}{2}
        && \Phi_b(x) &\mid 2(2x^{2t+3} - 2x^5 - 1) - x^5(4x^{2t-2} + 2x^5 - 3)\\
        \implies \quad && \Phi_b(x) &\mid -2x^{10} - x^5 - 2,
    \end{alignat*}
    which is not possible since the polynomial $2x^{10} + x^5 + 2$ has no roots of unity among its roots, as discussed earlier.

    In a similar fashion, if we suppose that $\Phi_b \mid W_t(x)$, this gives us
    \begin{align*}
        \Phi_b(x) &\mid -x^{2t+4} + 2x^{2t-1} - 2x,\\
        \Phi_b(x) &\mid 2x^{4t-1} - 2x^{2t+1} + x^{2t-4} .
    \end{align*}
    Due to $t \ge 6$, we swiftly obtain
    \begin{alignat}{2}
        \nonumber && \Phi_b(x) &\mid -x^{2t+4} + 2x^{2t-1} - 2x\\
        \nonumber \implies \quad && \Phi_b(x) &\mid -x(x^{2t+3} - 2x^{2t-2} + 2)\\
        \label{aux_16}\implies \quad && \Phi_b(x) &\mid x^{2t+3} - 2x^{2t-2} + 2
    \end{alignat}
    and
    \begin{alignat}{2}
        \nonumber && \Phi_b(x) &\mid 2x^{4t-1} - 2x^{2t+1} + x^{2t-4}\\
        \nonumber \implies \quad && \Phi_b(x) &\mid x^{2t-4}(2x^{2t+3} - 2x^5 + 1)\\
        \label{aux_17}\implies \quad && \Phi_b(x) &\mid 2x^{2t+3} - 2x^5 + 1 .
    \end{alignat}
    If we use Eqs.\ (\ref{aux_16}) and (\ref{aux_17}) together, we conclude that
    \begin{alignat}{2}
        \nonumber && \Phi_b(x) &\mid (2x^{2t+3} - 2x^5 + 1 ) - 2(x^{2t+3} - 2x^{2t-2} + 2)\\
        \label{aux_18} \implies \quad && \Phi_b(x) &\mid 4x^{2t-2} - 2x^5 - 3.
    \end{alignat}
    Furthermore, by using Eqs.\ (\ref{aux_17}) and (\ref{aux_18}) together, we get
    \begin{alignat*}{2}
        && \Phi_b(x) &\mid x^5(4x^{2t-2} - 2x^5 - 3) - 2(2x^{2t+3} - 2x^5 + 1)\\
        \implies \quad && \Phi_b(x) &\mid -2x^{10} + x^5 - 2,
    \end{alignat*}
    which is impossible given the fact that the polynomial $2x^{10} - x^5 + 2$ has no roots of unity among its roots, as we have already discussed.

    \medskip\noindent
    \emph{Subcase $t \equiv_5 3$}.\quad
    If we suppose that $\Phi_b(x) \mid U_t(x)$, we obtain
    \begin{align*}
        \Phi_b(x) &\mid x^{2t+4} + 2x^{2t-1},\\
        \Phi_b(x) &\mid -2x^{2t+1} - x^{2t-4}.
    \end{align*}
    Since $t \ge 3$, this directly gives us
    \begin{alignat}{2}
        \nonumber && \Phi_b(x) &\mid x^{2t+4} + 2x^{2t-1}\\
        \nonumber \implies \quad && \Phi_b(x) &\mid x^{2t-1}(x^5 + 2)\\
        \label{aux_19}\implies \quad && \Phi_b(x) &\mid x^5 + 2
    \end{alignat}
    and
    \begin{alignat}{2}
        \nonumber && \Phi_b(x) &\mid -2x^{2t+1} - x^{2t-4}\\
        \nonumber \implies \quad && \Phi_b(x) &\mid -x^{2t-4}(2x^5 + 1)\\
        \label{aux_20}\implies \quad && \Phi_b(x) &\mid 2x^5 + 1 .
    \end{alignat}
    By combining Eqs.\ (\ref{aux_19}) and (\ref{aux_20}), we reach
    \begin{alignat*}{2}
        && \Phi_b(x) &\mid 2(x^5 + 2) - (2x^5 + 1)\\
        \implies \quad && \Phi_b(x) &\mid 3,
    \end{alignat*}
    thus yielding a contradiction.

    On the other hand, if we suppose that $\Phi_b(x) \mid W_t(x)$, this gives us
    \begin{align*}
        \Phi_b(x) &\mid -x^{2t+4} + 2x^{2t-1},\\
        \Phi_b(x) &\mid -2x^{2t+1} + x^{2t-4},
    \end{align*}
    which further implies
    \begin{alignat}{2}
        \nonumber && \Phi_b(x) &\mid -x^{2t+4} + 2x^{2t-1}\\
        \nonumber \implies \quad && \Phi_b(x) &\mid -x^{2t-1}(x^5 - 2)\\
        \label{aux_21}\implies \quad && \Phi_b(x) &\mid x^5 - 2 ,
    \end{alignat}
    as well as
    \begin{alignat}{2}
        \nonumber && \Phi_b(x) &\mid -2x^{2t+1} + x^{2t-4}\\
        \nonumber \implies \quad && \Phi_b(x) &\mid -x^{2t-4}(2x^5 - 1)\\
        \label{aux_22}\implies \quad && \Phi_b(x) &\mid 2x^5 - 1 .
    \end{alignat}
    By combining Eqs.\ (\ref{aux_21}) and (\ref{aux_22}), we conclude that
    \begin{alignat*}{2}
        && \Phi_b(x) &\mid 2(x^5 - 2) - (2x^5 - 1)\\
        \implies \quad && \Phi_b(x) &\mid -3,
    \end{alignat*}
    which is not possible.

    \bigskip\noindent
    \emph{Case $p = 3$}.\quad
    This case can be resolved by taking into consideration the modular values given in Table \ref{remainders_3}. Thus, we divide the case into three subcases depending on the value of $t \bmod 3$.
    
    \begin{table}[h!t]
    {\footnotesize
    \begin{center}
    \begin{tabular}{rccc}
    \toprule & $t \equiv_3 0$ & $t \equiv_3 1$ & $t \equiv_3 2$ \\
    \midrule
    $1 \bmod 3$ & $1$ & $1$ & $1$\\
    $(2t-4) \bmod 3$ & $2$ & $1$ & $0$\\
    $(2t-1) \bmod 3$ & $2$ & $1$ & $0$\\
    $(2t+1) \bmod 3$ & $1$ & $0$ & $2$\\
    $(2t+4) \bmod 3$ & $1$ & $0$ & $2$\\
    $(4t-1) \bmod 3$ & $2$ & $0$ & $1$\\
    \bottomrule
    \end{tabular}
    \end{center}
    \caption{The elements of the set $M_t$ modulo $3$.}
    \label{remainders_3}
    }
    \end{table}

    \medskip\noindent
    \emph{Subcase $t \equiv_3 0$}.\quad
    In this subcase, if we suppose that $\Phi_b(x) \mid U_t(x)$, we then have
    \begin{align*}
        \Phi_b(x) &\mid x^{2t+4} - 2x^{2t+1} - 2x,\\
        \Phi_b(x) &\mid 2x^{4t-1} + 2x^{2t-1} - x^{2t-4},
    \end{align*}
    which immediately gives
    \begin{alignat}{2}
        \nonumber && \Phi_b(x) &\mid x^{2t+4} - 2x^{2t+1} - 2x\\
        \nonumber \implies \quad && \Phi_b(x) &\mid x(x^{2t+3} - 2x^{2t} - 2)\\
        \label{aux_23}\implies \quad && \Phi_b(x) &\mid x^{2t+3} - 2x^{2t} - 2
    \end{alignat}
    and
    \begin{alignat}{2}
        \nonumber && \Phi_b(x) &\mid 2x^{4t-1} + 2x^{2t-1} - x^{2t-4}\\
        \nonumber \implies \quad && \Phi_b(x) &\mid x^{2t-4}(2x^{2t+3} + 2x^3 - 1)\\
        \label{aux_24}\implies \quad && \Phi_b(x) &\mid 2x^{2t+3} + 2x^3 - 1 .
    \end{alignat}
    By using Eqs.\ (\ref{aux_23}) and (\ref{aux_24}) together, we obtain
    \begin{alignat}{2}
        \nonumber && \Phi_b(x) &\mid (2x^{2t+3} + 2x^3 - 1) - 2(x^{2t+3} - 2x^{2t} - 2)\\
        \label{aux_25} \implies \quad && \Phi_b(x) &\mid 4x^{2t} + 2x^3 + 3 .
    \end{alignat}
    Now, by using Eqs.\ (\ref{aux_24}) and (\ref{aux_25}), we get
    \begin{alignat*}{2}
        && \Phi_b(x) &\mid 2(2x^{2t+3} + 2x^3 - 1) - x^3(4x^{2t} + 2x^3 + 3)\\
        \implies \quad && \Phi_b(x) &\mid -2x^6 + x^3 - 2 .
    \end{alignat*}
    However, none of the roots of the polynomial $2x^6 - x^3 + 2$ are actually roots of unity, as shown in Apprendix \ref{problematic_polynomials}, which leads us to a contradiction.

    Similarly, if we suppose that $\Phi_b(x) \mid W_t(x)$, we obtain
    \begin{align*}
        \Phi_b(x) &\mid -x^{2t+4} - 2x^{2t+1} - 2x,\\
        \Phi_b(x) &\mid 2x^{4t-1} + 2x^{2t-1} + x^{2t-4},
    \end{align*}
    which quickly leads us to
    \begin{alignat}{2}
        \nonumber && \Phi_b(x) &\mid -x^{2t+4} - 2x^{2t+1} - 2x\\
        \nonumber \implies \quad && \Phi_b(x) &\mid -x(x^{2t+3} + 2x^{2t} + 2)\\
        \label{aux_26}\implies \quad && \Phi_b(x) &\mid x^{2t+3} + 2x^{2t} + 2
    \end{alignat}
    and
    \begin{alignat}{2}
        \nonumber && \Phi_b(x) &\mid 2x^{4t-1} + 2x^{2t-1} + x^{2t-4}\\
        \nonumber \implies \quad && \Phi_b(x) &\mid x^{2t-4}(2x^{2t+3} + 2x^3 + 1)\\
        \label{aux_27}\implies \quad && \Phi_b(x) &\mid 2x^{2t+3} + 2x^3 + 1 .
    \end{alignat}
    If we use Eqs.\ (\ref{aux_26}) and (\ref{aux_27}) together, we conclude that
    \begin{alignat}{2}
        \nonumber && \Phi_b(x) &\mid 2(x^{2t+3} + 2x^{2t} + 2) - (2x^{2t+3} + 2x^3 + 1)\\
        \label{aux_28} \implies \quad && \Phi_b(x) &\mid 4x^{2t} - 2x^3 + 3 .
    \end{alignat}
    Furthermore, by combining Eqs.\ (\ref{aux_27}) and (\ref{aux_28}), we reach
    \begin{alignat*}{2}
        && \Phi_b(x) &\mid 2(2x^{2t+3} + 2x^3 + 1) - x^3(4x^{2t} - 2x^3 + 3)\\
        \implies \quad && \Phi_b(x) &\mid 2x^6 + x^3 + 2 ,
    \end{alignat*}
    which leads to a contradiction, given the fact that the polynomial $2x^6 + x^3 + 2$ has no root which represents a root of unity, as demonstrated in Apprendix \ref{problematic_polynomials}.

    \medskip\noindent
    \emph{Subcase $t \equiv_3 1$}.\quad
    In this scenario, supposing that $\Phi_b(x) \mid U_t(x)$ is true leads to
    \begin{align*}
        \Phi_b(x) &\mid 2x^{2t-1} - x^{2t-4} - 2x,\\
        \Phi_b(x) &\mid 2x^{4t-1} + x^{2t+4} - 2x^{2t+1}.
    \end{align*}
    Furthermore, we swiftly get
    \begin{alignat}{2}
        \nonumber && \Phi_b(x) &\mid 2x^{2t-1} - x^{2t-4} - 2x\\
        \nonumber \implies \quad && \Phi_b(x) &\mid x(2x^{2t-2} - x^{2t-5} - 2)\\
        \label{aux_29}\implies \quad && \Phi_b(x) &\mid 2x^{2t-2} - x^{2t-5} - 2,
    \end{alignat}
    as well as
    \begin{alignat}{2}
        \nonumber && \Phi_b(x) &\mid 2x^{4t-1} + x^{2t+4} - 2x^{2t+1}\\
        \nonumber \implies \quad && \Phi_b(x) &\mid x^{2t+1}(2x^{2t-2} + x^3 - 2)\\
        \label{aux_30}\implies \quad && \Phi_b(x) &\mid 2x^{2t-2} + x^3 - 2 .
    \end{alignat}
    Given the fact that $t \ge 4$, subtracting the right-hand sides of Eqs.\ (\ref{aux_29}) and (\ref{aux_30}) helps us obtain
    \begin{alignat}{2}
        \nonumber && \Phi_b(x) &\mid (2x^{2t-2} + x^3 - 2) - (2x^{2t-2} - x^{2t-5} - 2)\\
        \nonumber \implies \quad && \Phi_b(x) &\mid x^{2t-5} + x^3\\
        \nonumber \implies \quad && \Phi_b(x) &\mid x^3 (x^{2t-8} + 1)\\
        \label{aux_31}\implies \quad && \Phi_b(x) &\mid x^{2t-8} + 1 .
    \end{alignat}
    Now, by combining Eqs.\ (\ref{aux_30}) and (\ref{aux_31}), we conclude that
    \begin{alignat*}{2}
        && \Phi_b(x) &\mid (2x^{2t-2} + x^3 - 2) - 2x^6(x^{2t-8} + 1)\\
        \implies \quad && \Phi_b(x) &\mid -2x^6 + x^3 - 2 ,
    \end{alignat*}
    thus yielding a contradiction due to the fact that the polynomial $2x^6 - x^3 + 2$ has no roots of unity among its roots, as discussed earlier.

    In a similar fashion, supposing that $\Phi_b(x) \mid W_t(x)$ gives us
    \begin{align*}
        \Phi_b(x) &\mid 2x^{2t-1} + x^{2t-4} - 2x,\\
        \Phi_b(x) &\mid 2x^{4t-1} - x^{2t+4} - 2x^{2t+1},
    \end{align*}
    which immediately leads to
    \begin{alignat}{2}
        \nonumber && \Phi_b(x) &\mid 2x^{2t-1} + x^{2t-4} - 2x\\
        \nonumber \implies \quad && \Phi_b(x) &\mid x(2x^{2t-2} + x^{2t-5} - 2)\\
        \label{aux_32}\implies \quad && \Phi_b(x) &\mid 2x^{2t-2} + x^{2t-5} - 2,
    \end{alignat}
    as well as
    \begin{alignat}{2}
        \nonumber && \Phi_b(x) &\mid 2x^{4t-1} - x^{2t+4} - 2x^{2t+1}\\
        \nonumber \implies \quad && \Phi_b(x) &\mid x^{2t+1}(2x^{2t-2} - x^3 - 2)\\
        \label{aux_33}\implies \quad && \Phi_b(x) &\mid 2x^{2t-2} - x^3 - 2 .
    \end{alignat}
    We can now subtract the right-hand sides of Eqs.\ (\ref{aux_32}) and (\ref{aux_33}) and use the fact that $t \ge 4$ in order to obtain
    \begin{alignat}{2}
        \nonumber && \Phi_b(x) &\mid (2x^{2t-2} + x^{2t-5} - 2) - (2x^{2t-2} - x^3 - 2)\\
        \nonumber \implies \quad && \Phi_b(x) &\mid x^{2t-5} + x^3\\
        \nonumber \implies \quad && \Phi_b(x) &\mid x^3 (x^{2t-8} + 1)\\
        \label{aux_34}\implies \quad && \Phi_b(x) &\mid x^{2t-8} + 1 .
    \end{alignat}
    Furthermore, combining Eqs.\ (\ref{aux_33}) and (\ref{aux_34}) gives us
    \begin{alignat*}{2}
        && \Phi_b(x) &\mid (2x^{2t-2} - x^3 - 2) - 2x^6(x^{2t-8} + 1)\\
        \implies \quad && \Phi_b(x) &\mid -2x^6 - x^3 - 2 ,
    \end{alignat*}
    which is impossible since the polynomial $2x^6 + x^3 + 2$ has no roots of unity among its roots, as we have already discussed.

    \medskip\noindent
    \emph{Subcase $t \equiv_3 2$}.\quad
    In this subcase, if we suppose that $\Phi_b(x) \mid U_t(x)$, then we directly obtain
    \begin{align*}
        \Phi_b(x) &\mid x^{2t+4} - 2x^{2t+1},\\
        \Phi_b(x) &\mid 2x^{2t-1} - x^{2t-4},
    \end{align*}
    which leads us to
    \begin{alignat}{2}
        \nonumber && \Phi_b(x) &\mid x^{2t+4} - 2x^{2t+1}\\
        \nonumber \implies \quad && \Phi_b(x) &\mid x^{2t+1}(x^3 - 2)\\
        \label{aux_35}\implies \quad && \Phi_b(x) &\mid x^3 - 2
    \end{alignat}
    and
    \begin{alignat}{2}
        \nonumber && \Phi_b(x) &\mid 2x^{2t-1} - x^{2t-4}\\
        \nonumber \implies \quad && \Phi_b(x) &\mid x^{2t-4}(2x^3 - 1)\\
        \label{aux_36}\implies \quad && \Phi_b(x) &\mid 2x^3 - 1 .
    \end{alignat}
    By using Eqs.\ (\ref{aux_35}) and (\ref{aux_36}) together, we reach
    \begin{alignat*}{2}
        && \Phi_b(x) &\mid (2x^3 - 1) - 2(x^3-2)\\
        \implies \quad && \Phi_b(x) &\mid 3,
    \end{alignat*}
    thus yielding a contradiction.

    Similarly, if we suppose that $\Phi_b(x) \mid W_t(x)$, we obtain
    \begin{align*}
        \Phi_b(x) &\mid -x^{2t+4} - 2x^{2t+1},\\
        \Phi_b(x) &\mid 2x^{2t-1} + x^{2t-4},
    \end{align*}
    which further gives
    \begin{alignat}{2}
        \nonumber && \Phi_b(x) &\mid -x^{2t+4} - 2x^{2t+1}\\
        \nonumber \implies \quad && \Phi_b(x) &\mid -x^{2t+1}(x^3 + 2)\\
        \label{aux_37}\implies \quad && \Phi_b(x) &\mid x^3 + 2,
    \end{alignat}
    as well as
    \begin{alignat}{2}
        \nonumber && \Phi_b(x) &\mid 2x^{2t-1} + x^{2t-4}\\
        \nonumber \implies \quad && \Phi_b(x) &\mid x^{2t-4}(2x^3 + 1)\\
        \label{aux_38}\implies \quad && \Phi_b(x) &\mid 2x^3 + 1 .
    \end{alignat}
    By combining Eqs.\ (\ref{aux_37}) and (\ref{aux_38}), we get
    \begin{alignat*}{2}
        && \Phi_b(x) &\mid (2x^3 + 1) - 2(x^3+2)\\
        \implies \quad && \Phi_b(x) &\mid -3,
    \end{alignat*}
    which is clearly not possible.
\end{proof}

If is clear that the formulation of Lemma \ref{square_free_2} is very similar to that of the previously proven Lemma \ref{square_free} from Section \ref{d1_family}. The key difference is that Lemma \ref{square_free_2} does not cover the case for $p = 2$. In fact, it can be shown that this property does not hold for the $U_t(x)$ and $W_t(x)$ polynomials, i.e.\ it is possible that they are divisible by a cyclotomic polynomial $\Phi_b(x)$ such that $4 \mid b$. However, although the condition $4 \mid b$ can hold, it can be shown that the only way for this to happen is if $b \in \{4, 8\}$, as demonstrated in the following lemma.

\begin{lemma}\label{almost_square_free}
    For any $t \ge 2$, if $U_t(x)$ or $W_t(x)$ are divisible by some cyclotomic polynomial $\Phi_b(x)$ such that $4 \mid b$, then $b = 4$ or $b = 8$.
\end{lemma}
\begin{proof}
    Suppose that $4 \mid b$. Given the fact that $2 \mid \frac{b}{2}$, it is clear that $\Phi_b(x) = \Phi_\frac{b}{2}(x^2)$ (see, for example, \cite[p.\ 160]{Nagell}). By using the same logic as in the proofs of Lemmas \ref{square_free} and \ref{square_free_2}, we obtain that $\Phi_b(x) \mid U_t(x)$ would imply
    \[
        \Phi_b(x) \mid U_t^{(j)}(x)
    \]
    for each $j = \overline{0, 1}$, where $U_t^{(j)}(x)$ represents the polynomial composed of the terms of $U_t(x)$ whose powers are congruent to $j$ modulo $2$. The same conclusion and notation can be applied to $W_t(x)$, as well.

    Since all the values $1, 2t-1, 2t+1, 4t-1$ are odd, while $2t + 4$ and $2t -4$ are even, we conclude that $\Phi_b \mid U_t(x)$ would immediately imply
    \[
        \Phi_b \mid x^{2t+4} - x^{2t-4} ,
    \]
    while $\Phi_b \mid W_t(x)$ would give
    \[
        \Phi_b \mid -x^{2t+4} + x^{2t-4} .
    \]
    Either way, we would reach
    \begin{alignat*}{2}
        && \Phi_b &\mid x^{2t+4} - x^{2t-4}\\
        \implies \quad && \Phi_b &\mid x^{2t-4}(x^8 - 1)\\
        \implies \quad && \Phi_b &\mid x^8 - 1.
    \end{alignat*}
    This practically means that each primitive $b$-th root of unity must be an eighth root of unity, i.e.\ $b \mid 8$. This is only possible if $b = 4$ or $b = 8$.
\end{proof}

A direct consequence of Lemmas \ref{square_free_2} and \ref{almost_square_free} is that if the cyclotomic polynomial $\Phi_b(x)$ divides $U_t(x)$ or $W_t(x)$, then the integer $b \in \mathbb{N}$ needs to either be square-free, or be equal to $4$ or $8$. In fact, for any $t \ge 2$, the only cyclotomic polynomials that could divide $U_t(x)$ or $W_t(x)$ are actually $\Phi_1(x), \Phi_2(x), \Phi_4(x), \Phi_8(x)$. This observation forms the central part of the proof of Theorem \ref{main_theorem_2}. However, in order to prove this claim, we will be in need of two more auxiliary lemmas, which hold a great resemblance to the previously proven Lemmas \ref{p_2p} and \ref{small_ones} from Section \ref{d1_family}.

\begin{lemma}\label{p_2p_2}
    For each $t \ge 2$ and each prime number $p \ge 7$, neither $U_t(x)$ nor $W_t(x)$ can be divisible by the cyclotomic polynomial $\Phi_p(x)$ or the cyclotomic polynomial $\Phi_{2p}(x)$.
\end{lemma}
\begin{proof}
    The proof can be done in an entirely analogous manner as the proof of Lemma \ref{p_2p} from Section \ref{d1_family}, the only difference being that Lemma \ref{unique_remainders_2} is used to conclude that there exists an element of the set $M_t$ that has a unique remainder within that set, instead of Lemma \ref{unique_remainder}. Thus, we choose to leave the proof out.
\end{proof}

\begin{lemma}\label{small_ones_2}
    For each $t \ge 2$ and each positive integer $b \in \{3, 5, 6, 10, 15, 30\}$, the cyclotomic polynomial $\Phi_b(x)$ divides neither $U_t(x)$ nor $W_t(x)$.
\end{lemma}
\begin{proof}
    The proof can be done by using an absolutely identical mechanism as the proof of Lemma \ref{small_ones} from Section \ref{d1_family}. We disclose the necessary computational results in Appendices \ref{appendix_ut} and \ref{appendix_wt}.
\end{proof}

We shall now formulate and prove the central lemma regarding the divisibility of the $U_t(x)$ and $W_t(x)$ polynomials by cyclotomic polynomials. Given the fact that these polynomials also have six non-zero terms, like the $Q_t(x)$ and $R_t(x)$ polynomials from Section \ref{d1_family}, using Theorem \ref{filaseta} becomes very convenient once more. The same prime number cancellation mechanism can be implemented for each prime number $p \ge 7$ that divides $b$ whenever $\Phi_b(x)$ divides either $U_t(x)$ or $W_t(x)$, as it was elaborated in Section \ref{d1_family}. Bearing this in mind, we give the following lemma.

\begin{lemma}\label{main_cyclotomic_2}
    For each $t \ge 2$, neither $U_t(x)$ nor $W_t(x)$ can be divisible by a cyclotomic polynomial $\Phi_b(x)$ such that $b \notin \{1, 2, 4, 8\}$.
\end{lemma}
\begin{proof}
    The proof will only be given for $U_t(x)$, given the fact that the proof regarding $W_t(x)$ is completely analogous. Suppose that $\Phi_b(x) \mid U_t(x)$, for some $b \in \mathbb{N}$ such that $b \notin \{1, 2, 4, 8 \}$. Since $b$ is not equal to $4$ or $8$, Lemmas \ref{square_free_2} and \ref{almost_square_free} tell us that $b$ must be a square-free integer. We now divide the problem into two separate cases, depending on whether $b$ is divisible by either $3$ or $5$, or not.

    \bigskip\noindent
    \emph{Case $3 \nmid b$ and $5 \nmid b$}.\quad
    This case is proved completely analogously to the corresponding case from the proof of Lemma \ref{main_cyclotomic} from Section \ref{d1_family}. The only small difference is that Lemma \ref{p_2p_2} is used instead of Lemma \ref{p_2p}.

    \bigskip\noindent
    \emph{Case $3 \mid b$ or $5 \mid b$}.\quad
    This case is also proved entirely analogously to the corresponding case from the proof of Lemma \ref{main_cyclotomic}. As expected, the only minor difference is that Lemma \ref{small_ones_2} is used instead of Lemma \ref{small_ones}.
\end{proof}

Taking into consideration all the results obtained in the previously disclosed lemmas, we are finally able to complete the proof of Theorem \ref{main_theorem_2}. Thus, we present the rest of the proof.

\bigskip\noindent
\emph{Proof of Theorem \ref{main_theorem_2}}.\quad
The case $t = 1$ is trivial to prove. Here, we simply have $S''_{1, n} = \left\{ \frac{n+2}{4}, \frac{n+6}{4} \right\}$, hence Eq.\ (\ref{polynomial_formula}) directly gives us
\begin{equation}\label{case_t1_2}
    P(\zeta) = \zeta^{\frac{n+2}{4}} + \dfrac{1}{\zeta^{\frac{n+2}{4}}} + \zeta^{\frac{n+6}{4}} + \dfrac{1}{\zeta^{\frac{n+6}{4}}} ,
\end{equation}
where $\zeta$ is an arbitrary $n$-th root of unity different from $1$ and $-1$. The condition $P(\zeta) = 0$ is clearly equivalent to
\begin{alignat*}{2}
    && P(\zeta) &= 0\\
    \iff \quad && \zeta^{\frac{n+6}{4}} \left( \zeta^{\frac{n+2}{4}} + \dfrac{1}{\zeta^{\frac{n+2}{4}}} + \zeta^{\frac{n+6}{4}} + \dfrac{1}{\zeta^{\frac{n+6}{4}}} \right) &=0\\
    \iff \quad && \zeta^{\frac{n}{2} + 3} + \zeta^{\frac{n}{2} + 2} + \zeta + 1 &= 0\\
    \iff \quad && (\zeta + 1)\left(\zeta^{\frac{n}{2} + 2} + 1 \right) &= 0\\
    \iff \quad && \zeta^{\frac{n}{2} + 2} + 1 &= 0 .
\end{alignat*}
Due to the fact that $\zeta^\frac{n}{2} \in \{1, -1\}$, it is obvious that $\zeta^{\frac{n}{2} + 2}$ is equal to either $\zeta^2 + 1$ or $-\zeta^2 + 1$. However, both of these expressions cannot be equal to zero. If $\zeta^2 + 1$ were to equal zero, then we would have that $\zeta \in \{i, -i\}$, which is impossible due to the fact that $n \equiv_4 2$, hence both $i$ and $-i$ are not $n$-th roots of unity. On the other hand, $-\zeta^2 + 1 \neq 0$ purely because $\zeta$ was chosen in such a way that it is distinct from both $1$ and $-1$. Thus, we get that $P(\zeta) \neq 0$ for any $n$-th root of unity different from $1$ and $-1$. By virtue of Lemma \ref{damnjanovic_lemma_1}, we conclude that $\mathcal{D}''_{1, n}$ truly is a circulant nut graph for any $n \ge 10$ such that $n \equiv_4 2$.

In the remainder of the proof, we will suppose that $t \ge 2$. Eq.\ (\ref{polynomial_formula}) now gives us
\[
    P(\zeta) = \left( \zeta^{\frac{n+2}{4}} + \frac{1}{\zeta^{\frac{n+2}{4}}} \right) + \left( \zeta^{\frac{n+6}{4}} + \frac{1}{\zeta^{\frac{n+6}{4}}} \right) + \sum_{j=1}^{t-1}\left( \zeta^j + \frac{1}{\zeta^j}\right) + \sum_{j=\frac{n}{2}-t+1}^{\frac{n}{2}-1}\left( \zeta^j + \frac{1}{\zeta^j}\right),
\]
where $\zeta$ is an arbitrarily chosen $n$-th root of unity different from $1$ and $-1$. It is not difficult to further obtain
\begin{equation}\label{polynomial_formula_4}
    P(\zeta) = \left( \zeta^{\frac{n+6}{4}} + \zeta^{\frac{n+2}{4}} + \frac{1}{\zeta^{\frac{n+2}{4}}} + \frac{1}{\zeta^{\frac{n+6}{4}}} \right) + \sum_{j=1}^{t-1}\left( \zeta^j + \frac{1}{\zeta^j} + \zeta^{\frac{n}{2} - j} + \frac{1}{\zeta^{\frac{n}{2}-j}} \right) .
\end{equation}
We will finalize the proof by demonstrating that $P(\zeta) \neq 0$ must be true. We will divide this proof into two cases in the same way as it was done while proving Theorem \ref{main_theorem_1}.

\bigskip\noindent
\emph{Case $\zeta^{\frac{n}{2}} = -1$}.\quad
In this case, it is straightforward to notice that $\zeta^{\frac{n}{2} - j} = -\dfrac{1}{\zeta^j}$ and $\dfrac{1}{\zeta^{\frac{n}{2} - j}} = -\zeta^j$, which directly leads us to
\[
    \zeta^j + \frac{1}{\zeta^j} + \zeta^{\frac{n}{2} - j} + \frac{1}{\zeta^{\frac{n}{2}-j}} = 0
\]
for any $j = \overline{1, t-1}$. Hence, it is convenient to simplify Eq.\ (\ref{polynomial_formula_4}) in order to get
\[
    P(\zeta) = \zeta^{\frac{n+6}{4}} + \zeta^{\frac{n+2}{4}} + \frac{1}{\zeta^{\frac{n+2}{4}}} + \frac{1}{\zeta^{\frac{n+6}{4}}} .
\]
However, this formula is identical to Eq.\ (\ref{case_t1_2}), which we had to deal with while solving the case $t = 1$. An almost absolutely identical proof can be used in order to show that $P(\zeta) \neq 0$ must hold, as desired, hence we choose to leave it out.

\bigskip\noindent
\emph{Case $\zeta^{\frac{n}{2}} = 1$}.\quad
It is easy to see that $\zeta^{\frac{n}{2} - j} = \dfrac{1}{\zeta^j}$ and $\dfrac{1}{\zeta^{\frac{n}{2} - j}} = \zeta^j$. This implies
\[
    \zeta^j + \frac{1}{\zeta^j} + \zeta^{\frac{n}{2} - j} + \frac{1}{\zeta^{\frac{n}{2}-j}} = 2 \left( \zeta^j + \frac{1}{\zeta^j} \right)
\]
for any $j = \overline{1, t-1}$. According to Eq.\ (\ref{polynomial_formula_4}), the equality $P(\zeta) = 0$ becomes equivalent to
\begin{alignat*}{2}
   && P(\zeta) &= 0\\
    \iff \quad && \left( \zeta^{\frac{n+6}{4}} + \zeta^{\frac{n+2}{4}} + \frac{1}{\zeta^{\frac{n+2}{4}}} + \frac{1}{\zeta^{\frac{n+6}{4}}} \right) + 2 \sum_{j=1}^{t-1}\left( \zeta^j + \frac{1}{\zeta^j} \right) &= 0\\
    \iff \quad && \left( \zeta^{\frac{n+6}{4}} + \zeta^{\frac{n+2}{4}} + \frac{1}{\zeta^{\frac{n+2}{4}}} + \frac{1}{\zeta^{\frac{n+6}{4}}} \right) - 2 + 2 \sum_{j=1-t}^{t-1} \zeta^j &= 0\\
    \iff \quad && \zeta^{t-1} \left( \left( \zeta^{\frac{n+6}{4}} + \zeta^{\frac{n+2}{4}} + \frac{1}{\zeta^{\frac{n+2}{4}}} + \frac{1}{\zeta^{\frac{n+6}{4}}} \right) - 2 + 2 \sum_{j=1-t}^{t-1} \zeta^j \right) &= 0\\
    \iff \quad && \left( \zeta^{t+\frac{n+2}{4}} + \zeta^{t+\frac{n-2}{4}} + \zeta^{t-\frac{n+6}{4}} + \zeta^{t-\frac{n+10}{4}} \right) - 2\zeta^{t-1} + 2 \sum_{j=0}^{2t-2} \zeta^j &= 0\\
    \iff \quad && (\zeta - 1)\left( \zeta^{t+\frac{n+2}{4}} + \zeta^{t+\frac{n-2}{4}} + \zeta^{t-\frac{n+6}{4}} + \zeta^{t-\frac{n+10}{4}} - 2\zeta^{t-1} + 2 \sum_{j=0}^{2t-2} \zeta^j \right) &= 0\\
    \iff \quad && \zeta^{t+\frac{n+6}{4}} - \zeta^{t+\frac{n-2}{4}} + \zeta^{t-\frac{n+2}{4}} - \zeta^{t-\frac{n+10}{4}} - 2\zeta^t + 2\zeta^{t-1} + 2 \zeta^{2t-1} - 2 &= 0\\
    \iff \quad && 2 \zeta^{2t-1} + \zeta^{t+\frac{n+6}{4}} - \zeta^{t+\frac{n-2}{4}} - 2\zeta^t + 2\zeta^{t-1} + \zeta^{t-\frac{n+2}{4}} - \zeta^{t-\frac{n+10}{4}} - 2 &= 0 .
\end{alignat*}

Now, if we define $\psi$ to be one of the two possible square roots of $\zeta$, i.e.\ a complex number such that $\psi^2 = \zeta$, then it is easy to see that the condition $P(\zeta) = 0$ becomes equivalent to
\begin{equation}\label{new_formula_psi}
    2\psi^{4t - 2} + \psi^{2t + 3 + \frac{n}{2}} - \psi^{2t-1+\frac{n}{2}} - 2\psi^{2t} + 2\psi^{2t-2} + \psi^{2t-1-\frac{n}{2}} - \psi^{2t - 5 - \frac{n}{2}} - 2 = 0 .
\end{equation}
Taking into consideration that the condition $\zeta^{\frac{n}{2}} = 1$ translates to $\psi^n = 1$, it becomes clear that in order to prove that $P(\zeta) \neq 0$ for any $n$-th root of unity $\zeta$ such that $\zeta \neq 1, -1$ and $\zeta^{\frac{n}{2}} = 1$, it is enough to show that Eq.\ (\ref{new_formula_psi}) does not hold for any $n$-th root of unity $\psi$ such that $\psi \notin \{1, -1, i, -i \}$. Since $i$ and $-i$ are not $n$-th roots of unity to begin with, due to $n \equiv_4 2$, it suffices to prove that Eq.\ (\ref{new_formula_psi}) has no $n$-th roots of unity among its roots, besides potentially $1$ and $-1$.

We now divide the case into two subcases, depending on whether $\psi^\frac{n}{2}$ is equal to $1$ or $-1$.

\medskip\noindent
\emph{Subcase $\psi^\frac{n}{2} = 1$}.\quad
Here, it is straightforward to deduct that Eq.\ (\ref{new_formula_psi}) is equivalent to
\begin{alignat*}{2}
    && 2\psi^{4t - 2} + \psi^{2t + 3} - \psi^{2t-1} - 2\psi^{2t} + 2\psi^{2t-2} + \psi^{2t-1} - \psi^{2t - 5} - 2 &= 0\\
    \iff \quad && 2\psi^{4t - 2} + \psi^{2t + 3} - 2\psi^{2t} + 2\psi^{2t-2} - \psi^{2t - 5} - 2 &= 0\\
    \iff \quad && 2\psi^{4t - 1} + \psi^{2t + 4} - 2\psi^{2t+1} + 2\psi^{2t-1} - \psi^{2t - 4} - 2\psi &= 0.
\end{alignat*}
Now, suppose that Eq.\ (\ref{new_formula_psi}) does hold for some $n$-th root of unity $\psi$ different from $1$ and $-1$. It is clear that $\psi$ must be a primitive $b$-th root of unity for some $b \neq 1, 2$, since $\psi \neq 1, -1$. Also, due to the fact that $b \mid n$ and $n \equiv_4 2$, we conclude that $b \neq 4, 8$. However, Eq.\ (\ref{new_formula_psi}) directly implies that $\psi$ is a root of the polynomial $U_t(x)$. Hence, we get that $U_t(x)$ is divisible by a cyclotomic polynomial $\Phi_b(x)$ such that $b \neq 1, 2, 4, 8$, which is not possible according to Lemma \ref{main_cyclotomic_2}, thus yielding a contradiction.

\medskip\noindent
\emph{Subcase $\psi^\frac{n}{2} = -1$}.\quad
In this subcase, Eq.\ (\ref{new_formula_psi}) quickly becomes equivalent to
\begin{alignat*}{2}
    && 2\psi^{4t - 2} - \psi^{2t + 3} + \psi^{2t-1} - 2\psi^{2t} + 2\psi^{2t-2} - \psi^{2t-1} + \psi^{2t - 5} - 2 &= 0\\
    \iff \quad && 2\psi^{4t - 2} - \psi^{2t + 3} - 2\psi^{2t} + 2\psi^{2t-2} + \psi^{2t - 5} - 2 &= 0\\
    \iff \quad && 2\psi^{4t - 1} - \psi^{2t + 4} - 2\psi^{2t+1} + 2\psi^{2t-1} + \psi^{2t - 4} - 2\psi &= 0.
\end{alignat*}
If we suppose that Eq.\ (\ref{new_formula_psi}) is true for some $n$-th root of unity $\psi$ different from $1$ and $-1$, we can use the same logic implemented in the previous subcase in order to immediately obtain that $\psi$ has to be a primitive $b$-th root of unity for some $b \notin \{1, 2, 4, 8\}$. However, Eq.\ (\ref{new_formula_psi}) implies that $\psi$ is a root of the polynomial $W_t(x)$, which further means that this polynomial has to be divisible by a cyclotomic polynomial $\Phi_b(x)$ such that $b \notin \{1, 2, 4, 8 \}$. By virtue of Lemma \ref{main_cyclotomic_2}, this is impossible, which completes the proof. \hfill\qed

\section{Conclusion}\label{conclusion}

In conclusion, Theorems \ref{main_theorem_1} and \ref{main_theorem_2} provide a construction for any $4t$-regular circulant nut graph of order $n$, whenever
\begin{itemize}
    \item $t$ is odd and $n \ge 4t + 4$;
    \item $t$ is even and $n \ge 4t + 6$ and such that $n \equiv_4 2$.
\end{itemize}
For the odd values of $t$, these two theorems provide a way to construct a $4t$-regular circulant nut graph of every possible order, thereby fully resolving the circulant nut graph order--degree existence problem. On the other hand, for the even values of $t$, Theorem~\ref{main_cyclotomic_2} shows that there exists a $4t$-regular circulant nut graph of each order $n \ge 4t + 6$ such that $n \equiv_4 2$, without covering the case when $4 \mid n$.

Damnjanovi\'c and Stevanovi\'c \cite[Proposition 19]{Damnjanovic} have shown that an interesting irregularity exists for the case $t=2$. Moreover, there does not exist an $8$-regular circulant nut graph of order $16$. In fact, the set of all the values $n$ such that there exists an $8$-regular circulant nut graph of order $n$ is given by the expression $\{14\} \cup \{18, 20, 22, 24, 26, \ldots \}$.

Despite the irregularity that occurs for $t = 2$, the experimental results obtained in \cite{Damnjanovic} dictate that for each even integer $t$ such that $4 \le t \le 1300$ there does exist a $4t$-regular circulant nut graph of order $n$ for each even $n \ge 4t + 6$. Bearing this mind, we end the paper with the following conjecture.

\begin{conjecture}
    For each even $t \ge 4$ and each $n \ge 4t + 8$ divisible by four, there exists a $4t$-regular circulant nut graph of order $n$.
\end{conjecture}

\appendix

\section{\texorpdfstring{$Q_t^{\bmod b}(x) \bmod{\Phi_b(x)}$}{Qt remainders} table}\label{appendix_qt}

{\scriptsize
\begin{longtable}{lll}
\toprule $t \bmod 3$ & $Q_t^{\bmod 3}(x)$ & $Q_t^{\bmod 3}(x) \bmod{\Phi_3(x)}$ \\
\midrule
$0$ & $-3+3 x^2$ & $-6-3 x$ \\
$1$ & $-1+x$ & $-1+x$ \\
$2$ & $x-x^2$ & $1+2 x$ \\
\midrule $t \bmod 5$ & $Q_t^{\bmod 5}(x)$ & $Q_t^{\bmod 5}(x) \bmod{\Phi_5(x)}$ \\ 
\midrule
$0$ & $-3+x-x^3+3 x^4$ & $-6-2 x-3 x^2-4 x^3$ \\
$1$ & $-1+x+x^2-x^4$ & $2 x+2 x^2+x^3$ \\
$2$ & $-3+x-x^2+3 x^3$ & $-3+x-x^2+3 x^3$ \\
$3$ & $-x+x^2-x^3+x^4$ & $-1-2 x-2 x^3$ \\
$4$ & $-1+x^2+x^3-x^4$ & $x+2 x^2+2 x^3$ \\
\midrule $t \bmod 6$ & $Q_t^{\bmod 6}(x)$ & $Q_t^{\bmod 6}(x) \bmod{\Phi_6(x)}$ \\ 
\midrule
$0$ & $-3+x-x^4+3 x^5$ & $-x$ \\
$1$ & $-1+x+x^2-x^5$ & $-3+3 x$ \\
$2$ & $-3+x-x^2+3 x^3$ & $-5$ \\
$3$ & $-2-x+x^2-x^3+x^4+2 x^5$ & $-3 x$ \\
$4$ & $-2+2 x-x^2+x^3-x^4+x^5$ & $-1+x$ \\
$5$ & $-1+x^3+x^4-x^5$ & $-3$ \\
\midrule $t \bmod 10$ & $Q_t^{\bmod 10}(x)$ & $Q_t^{\bmod 10}(x) \bmod{\Phi_{10}(x)}$ \\
\midrule
$0$ & $-3+x-x^8+3 x^9$ & $-2 x+3 x^2-2 x^3$ \\
$1$ & $-1+x+x^2-x^9$ & $-2+2 x+x^3$ \\
$2$ & $-3+x-x^2+3 x^3$ & $-3+x-x^2+3 x^3$ \\
$3$ & $-2-x+x^2-x^3+x^4+2 x^5$ & $-5$ \\
$4$ & $-2-x^2+x^3-x^4+x^5+2 x^7$ & $-2-x-2 x^2$ \\
$5$ & $-2-x^3+x^4-x^5+x^6+2 x^9$ & $-2 x+x^2-2 x^3$ \\
$6$ & $-2+2 x-x^4+x^5-x^6+x^7$ & $-2+2 x-x^3$ \\
$7$ & $-2+2 x^3-x^5+x^6-x^7+x^8$ & $-1-x+x^2+x^3$ \\
$8$ & $-2+2 x^5-x^6+x^7-x^8+x^9$ & $-3$ \\
$9$ & $-1+x^7+x^8-x^9$ & $-2+x-2 x^2$ \\
\midrule $t \bmod 15$ & $Q_t^{\bmod 15}(x)$ & $Q_t^{\bmod 15}(x) \bmod{\Phi_{15}(x)}$ \\ 
\midrule
$0$ & $-3+x-x^{13}+3 x^{14}$ & $-1+2 x-3 x^2+3 x^3-2 x^4-x^5+3 x^6-2 x^7$ \\
$1$ & $-1+x+x^2-x^{14}$ & $-2+x+2 x^2-x^3+x^4-x^6+x^7$ \\
$2$ & $-3+x-x^2+3 x^3$ & $-3+x-x^2+3 x^3$ \\
$3$ & $-2-x+x^2-x^3+x^4+2 x^5$ & $-2-x+x^2-x^3+x^4+2 x^5$ \\
$4$ & $-2-x^2+x^3-x^4+x^5+2 x^7$ & $-2-x^2+x^3-x^4+x^5+2 x^7$ \\
$5$ & $-2-x^3+x^4-x^5+x^6+2 x^9$ & $-4+2 x^2-3 x^3+x^4-x^5-x^6+2 x^7$ \\
$6$ & $-2-x^4+x^5-x^6+x^7+2 x^{11}$ & $-2-2 x-x^4+x^5-3 x^6+x^7$ \\
$7$ & $-2-x^5+x^6-x^7+x^8+2 x^{13}$ & $-1-x-x^3-x^4+x^6-2 x^7$ \\
$8$ & $-x^6+x^7-x^8+x^9$ & $-x+x^2-x^4+x^5-2 x^6+x^7$ \\
$9$ & $-2+2 x^2-x^7+x^8-x^9+x^{10}$ & $-3+x+x^2+x^4-2 x^5+x^6-x^7$ \\
$10$ & $-2+2 x^4-x^8+x^9-x^{10}+x^{11}$ & $-1-2 x+x^2+x^4+2 x^5-2 x^6$ \\
$11$ & $-2+2 x^6-x^9+x^{10}-x^{11}+x^{12}$ & $-2+x-2 x^2+x^3-x^5+4 x^6-2 x^7$ \\
$12$ & $-2+2 x^8-x^{10}+x^{11}-x^{12}+x^{13}$ & $-2+x^2-2 x^3+x^4-x^6+2 x^7$ \\
$13$ & $-2+2 x^{10}-x^{11}+x^{12}-x^{13}+x^{14}$ & $-4+2 x-2 x^2+x^3-3 x^5+2 x^6-x^7$ \\
$14$ & $-1+x^{12}+x^{13}-x^{14}$ & $-1-x-x^3+x^5-x^6-x^7$ \\
\midrule $t \bmod 30$ & $Q_t^{\bmod 30}(x)$ & $Q_t^{\bmod 30}(x) \bmod{\Phi_{30}(x)}$ \\ 
\midrule
$0$ & $-3+x-x^{28}+3 x^{29}$ & $-7+3 x^2+3 x^3+4 x^4+x^5-3 x^6-4 x^7$ \\
$1$ & $-1+x+x^2-x^{29}$ & $x-x^3-x^4+x^6+x^7$ \\
$2$ & $-3+x-x^2+3 x^3$ & $-3+x-x^2+3 x^3$ \\
$3$ & $-2-x+x^2-x^3+x^4+2 x^5$ & $-2-x+x^2-x^3+x^4+2 x^5$ \\
$4$ & $-2-x^2+x^3-x^4+x^5+2 x^7$ & $-2-x^2+x^3-x^4+x^5+2 x^7$ \\
$5$ & $-2-x^3+x^4-x^5+x^6+2 x^9$ & $-2 x^2-3 x^3+x^4-x^5+3 x^6+2 x^7$ \\
$6$ & $-2-x^4+x^5-x^6+x^7+2 x^{11}$ & $-2-2 x-x^4+x^5+x^6+x^7$ \\
$7$ & $-2-x^5+x^6-x^7+x^8+2 x^{13}$ & $-5-3 x+x^3+3 x^4+2 x^5+x^6-4 x^7$ \\
$8$ & $-2-x^6+x^7-x^8+x^9+2 x^{15}$ & $-2+x-x^2-2 x^3-x^4-x^5+3 x^7$ \\
$9$ & $-2-x^7+x^8-x^9+x^{10}+2 x^{17}$ & $-5-x-x^2+2 x^3+x^4+2 x^5-x^6-3 x^7$ \\
$10$ & $-2-x^8+x^9-x^{10}+x^{11}+2 x^{19}$ & $1-x^2-2 x^3-3 x^4-2 x^5+2 x^6+2 x^7$ \\
$11$ & $-2-x^9+x^{10}-x^{11}+x^{12}+2 x^{21}$ & $-4+x+x^3+x^5-4 x^6$ \\
$12$ & $-2-x^{10}+x^{11}-x^{12}+x^{13}+2 x^{23}$ & $x^2-2 x^3-x^4-2 x^5+x^6$ \\
$13$ & $-2-x^{11}+x^{12}-x^{13}+x^{14}+2 x^{25}$ & $2+2 x-2 x^2-x^3-2 x^4-3 x^5+3 x^7$ \\
$14$ & $-2-x^{12}+x^{13}-x^{14}+x^{15}+2 x^{27}$ & $-5-x+4 x^2+x^3+2 x^4+x^5-x^6-5 x^7$ \\
$15$ & $-2-x^{13}+x^{14}-x^{15}+x^{16}+2 x^{29}$ & $-1+x^2+x^3-x^5-x^6$ \\
$16$ & $-2+2 x-x^{14}+x^{15}-x^{16}+x^{17}$ & $-4+3 x+x^3+x^4-x^6-x^7$ \\
$17$ & $-2+2 x^3-x^{15}+x^{16}-x^{17}+x^{18}$ & $-1-x+x^2+x^3$ \\
$18$ & $-2+2 x^5-x^{16}+x^{17}-x^{18}+x^{19}$ & $-2+x-x^2+x^3-x^4+2 x^5$ \\
$19$ & $-2+2 x^7-x^{17}+x^{18}-x^{19}+x^{20}$ & $-2+x^2-x^3+x^4-x^5+2 x^7$ \\
$20$ & $-2+2 x^9-x^{18}+x^{19}-x^{20}+x^{21}$ & $-2 x^2-x^3-x^4+x^5+x^6+2 x^7$ \\
$21$ & $-2+2 x^{11}-x^{19}+x^{20}-x^{21}+x^{22}$ & $-2-2 x+x^4-x^5+3 x^6-x^7$ \\
$22$ & $-2+2 x^{13}-x^{20}+x^{21}-x^{22}+x^{23}$ & $-3-x-x^3+x^4+2 x^5-x^6$ \\
$23$ & $-2+2 x^{15}-x^{21}+x^{22}-x^{23}+x^{24}$ & $-6-x+x^2+2 x^3+x^4+x^5-3 x^7$ \\
$24$ & $-2+2 x^{17}-x^{22}+x^{23}-x^{24}+x^{25}$ & $1+x-3 x^2-2 x^3-x^4-2 x^5+x^6+3 x^7$ \\
$25$ & $-2+2 x^{19}-x^{23}+x^{24}-x^{25}+x^{26}$ & $-5+x^2+2 x^3-x^4+2 x^5-2 x^6-2 x^7$ \\
$26$ & $-2+2 x^{21}-x^{24}+x^{25}-x^{26}+x^{27}$ & $-x-x^3-x^5$ \\
$27$ & $-2+2 x^{23}-x^{25}+x^{26}-x^{27}+x^{28}$ & $4 x-x^2-2 x^3-3 x^4-2 x^5-x^6+4 x^7$ \\
$28$ & $-2+2 x^{25}-x^{26}+x^{27}-x^{28}+x^{29}$ & $-2-2 x+2 x^2+x^3+2 x^4-x^5-3 x^7$ \\
$29$ & $-1+x^{27}+x^{28}-x^{29}$ & $1+x-x^3-2 x^4-x^5+x^6+x^7$ \\
\bottomrule
\end{longtable}
}

\section{\texorpdfstring{$R_t^{\bmod b}(x) \bmod{\Phi_b(x)}$}{Rt remainders} table}\label{appendix_rt}

{\scriptsize
\begin{longtable}{lll}
\toprule $t \bmod 3$ & $R_t^{\bmod 3}(x)$ & $R_t^{\bmod 3}(x) \bmod{\Phi_3(x)}$ \\
\midrule
$0$ & $-5+5 x^2$ & $-10-5 x$ \\
$1$ & $1-x$ & $1-x$ \\
$2$ & $3 x-3 x^2$ & $3+6 x$ \\
\midrule $t \bmod 5$ & $R_t^{\bmod 5}(x)$ & $R_t^{\bmod 5}(x) \bmod{\Phi_5(x)}$ \\ 
\midrule
$0$ & $-5-x+x^3+5 x^4$ & $-10-6 x-5 x^2-4 x^3$ \\
$1$ & $1-x-x^2+x^4$ & $-2 x-2 x^2-x^3$ \\
$2$ & $-1+3 x-3 x^2+x^3$ & $-1+3 x-3 x^2+x^3$ \\
$3$ & $x+3 x^2-3 x^3-x^4$ & $1+2 x+4 x^2-2 x^3$ \\
$4$ & $-3+3 x^2+3 x^3-3 x^4$ & $3 x+6 x^2+6 x^3$ \\
\midrule $t \bmod 6$ & $R_t^{\bmod 6}(x)$ & $R_t^{\bmod 6}(x) \bmod{\Phi_6(x)}$ \\ 
\midrule
$0$ & $-5-x+x^4+5 x^5$ & $-7 x$ \\
$1$ & $1-x-x^2+x^5$ & $3-3 x$ \\
$2$ & $-1+3 x-3 x^2+x^3$ & $1$ \\
$3$ & $-2+x+3 x^2-3 x^3-x^4+2 x^5$ & $3 x$ \\
$4$ & $-2+2 x+x^2+3 x^3-3 x^4-x^5$ & $-7+7 x$ \\
$5$ & $-3+3 x^3+3 x^4-3 x^5$ & $-9$ \\
\midrule $t \bmod 10$ & $R_t^{\bmod 10}(x)$ & $R_t^{\bmod 10}(x) \bmod{\Phi_{10}(x)}$ \\
\midrule
$0$ & $-5-x+x^8+5 x^9$ & $-6 x+5 x^2-6 x^3$ \\
$1$ & $1-x-x^2+x^9$ & $2-2 x-x^3$ \\
$2$ & $-1+3 x-3 x^2+x^3$ & $-1+3 x-3 x^2+x^3$ \\
$3$ & $-2+x+3 x^2-3 x^3-x^4+2 x^5$ & $-3+4 x^2-4 x^3$ \\
$4$ & $-2+x^2+3 x^3-3 x^4-x^5+2 x^7$ & $2-3 x+2 x^2$ \\
$5$ & $-2+x^3+3 x^4-3 x^5-x^6+2 x^9$ & $2 x-x^2+2 x^3$ \\
$6$ & $-2+2 x+x^4+3 x^5-3 x^6-x^7$ & $-6+6 x+x^3$ \\
$7$ & $-2+2 x^3+x^5+3 x^6-3 x^7-x^8$ & $-3-3 x+3 x^2+3 x^3$ \\
$8$ & $-2+2 x^5+x^6+3 x^7-3 x^8-x^9$ & $-5-4 x^2+4 x^3$ \\
$9$ & $-3+3 x^7+3 x^8-3 x^9$ & $-6+3 x-6 x^2$ \\
\midrule $t \bmod 15$ & $R_t^{\bmod 15}(x)$ & $R_t^{\bmod 15}(x) \bmod{\Phi_{15}(x)}$ \\ 
\midrule
$0$ & $-5-x+x^{13}+5 x^{14}$ & $1-2 x-5 x^2+5 x^3-6 x^4+x^5+5 x^6-6 x^7$ \\
$1$ & $1-x-x^2+x^{14}$ & $2-x-2 x^2+x^3-x^4+x^6-x^7$ \\
$2$ & $-1+3 x-3 x^2+x^3$ & $-1+3 x-3 x^2+x^3$ \\
$3$ & $-2+x+3 x^2-3 x^3-x^4+2 x^5$ & $-2+x+3 x^2-3 x^3-x^4+2 x^5$ \\
$4$ & $-2+x^2+3 x^3-3 x^4-x^5+2 x^7$ & $-2+x^2+3 x^3-3 x^4-x^5+2 x^7$ \\
$5$ & $-2+x^3+3 x^4-3 x^5-x^6+2 x^9$ & $-4+2 x^2-x^3+3 x^4-3 x^5-3 x^6+2 x^7$ \\
$6$ & $-2+x^4+3 x^5-3 x^6-x^7+2 x^{11}$ & $-2-2 x+x^4+3 x^5-5 x^6-x^7$ \\
$7$ & $-2+x^5+3 x^6-3 x^7-x^8+2 x^{13}$ & $1-3 x+x^3-3 x^4+4 x^5+3 x^6-6 x^7$ \\
$8$ & $x^6+3 x^7-3 x^8-x^9$ & $4-3 x-x^2+4 x^3-3 x^4+3 x^5+2 x^6-x^7$ \\
$9$ & $-2+2 x^2+x^7+3 x^8-3 x^9-x^{10}$ & $-1+3 x-x^2+3 x^4-2 x^5+3 x^6+x^7$ \\
$10$ & $-2+2 x^4+x^8+3 x^9-3 x^{10}-x^{11}$ & $-3+2 x+3 x^2-4 x^3+3 x^4+2 x^5-2 x^6+4 x^7$ \\
$11$ & $-2+2 x^6+x^9+3 x^{10}-3 x^{11}-x^{12}$ & $-6+3 x+2 x^2-x^3-3 x^5+4 x^6+2 x^7$ \\
$12$ & $-2+2 x^8+x^{10}+3 x^{11}-3 x^{12}-x^{13}$ & $-6+3 x^2-2 x^3+3 x^4-4 x^5-3 x^6+6 x^7$ \\
$13$ & $-2+2 x^{10}+x^{11}+3 x^{12}-3 x^{13}-x^{14}$ & $-8+2 x-2 x^2-x^3+4 x^4-5 x^5-2 x^6+x^7$ \\
$14$ & $-3+3 x^{12}+3 x^{13}-3 x^{14}$ & $-3-3 x-3 x^3+3 x^5-3 x^6-3 x^7$ \\
\midrule $t \bmod 30$ & $R_t^{\bmod 30}(x)$ & $R_t^{\bmod 30}(x) \bmod{\Phi_{30}(x)}$ \\ 
\midrule
$0$ & $-5-x+x^{28}+5 x^{29}$ & $-9+5 x^2+5 x^3+4 x^4-x^5-5 x^6-4 x^7$ \\
$1$ & $1-x-x^2+x^{29}$ & $-x+x^3+x^4-x^6-x^7$ \\
$2$ & $-1+3 x-3 x^2+x^3$ & $-1+3 x-3 x^2+x^3$ \\
$3$ & $-2+x+3 x^2-3 x^3-x^4+2 x^5$ & $-2+x+3 x^2-3 x^3-x^4+2 x^5$ \\
$4$ & $-2+x^2+3 x^3-3 x^4-x^5+2 x^7$ & $-2+x^2+3 x^3-3 x^4-x^5+2 x^7$ \\
$5$ & $-2+x^3+3 x^4-3 x^5-x^6+2 x^9$ & $-2 x^2-x^3+3 x^4-3 x^5+x^6+2 x^7$ \\
$6$ & $-2+x^4+3 x^5-3 x^6-x^7+2 x^{11}$ & $-2-2 x+x^4+3 x^5-x^6-x^7$ \\
$7$ & $-2+x^5+3 x^6-3 x^7-x^8+2 x^{13}$ & $-3-x-x^3+x^4+2 x^5+3 x^6-4 x^7$ \\
$8$ & $-2+x^6+3 x^7-3 x^8-x^9+2 x^{15}$ & $-2+3 x+x^2-2 x^3-3 x^4-3 x^5+5 x^7$ \\
$9$ & $-2+x^7+3 x^8-3 x^9-x^{10}+2 x^{17}$ & $-7-3 x+x^2+6 x^3+3 x^4+2 x^5-3 x^6-5 x^7$ \\
$10$ & $-2+x^8+3 x^9-3 x^{10}-x^{11}+2 x^{19}$ & $3-3 x^2-2 x^3-x^4-2 x^5+2 x^6+2 x^7$ \\
$11$ & $-2+x^9+3 x^{10}-3 x^{11}-x^{12}+2 x^{21}$ & $-4+3 x-x^3+3 x^5-4 x^6$ \\
$12$ & $-2+x^{10}+3 x^{11}-3 x^{12}-x^{13}+2 x^{23}$ & $3 x^2-2 x^3-3 x^4-2 x^5+3 x^6$ \\
$13$ & $-2+x^{11}+3 x^{12}-3 x^{13}-x^{14}+2 x^{25}$ & $2+2 x-2 x^2+x^3-2 x^4-5 x^5+5 x^7$ \\
$14$ & $-2+x^{12}+3 x^{13}-3 x^{14}-x^{15}+2 x^{27}$ & $-7-3 x+4 x^2+3 x^3+6 x^4+3 x^5-3 x^6-7 x^7$ \\
$15$ & $-2+x^{13}+3 x^{14}-3 x^{15}-x^{16}+2 x^{29}$ & $1-x^2-x^3+x^5+x^6$ \\
$16$ & $-2+2 x+x^{14}+3 x^{15}-3 x^{16}-x^{17}$ & $-4+5 x-x^3-x^4+x^6+x^7$ \\
$17$ & $-2+2 x^3+x^{15}+3 x^{16}-3 x^{17}-x^{18}$ & $-3-3 x+3 x^2+3 x^3$ \\
$18$ & $-2+2 x^5+x^{16}+3 x^{17}-3 x^{18}-x^{19}$ & $-2-x-3 x^2+3 x^3+x^4+2 x^5$ \\
$19$ & $-2+2 x^7+x^{17}+3 x^{18}-3 x^{19}-x^{20}$ & $-2-x^2-3 x^3+3 x^4+x^5+2 x^7$ \\
$20$ & $-2+2 x^9+x^{18}+3 x^{19}-3 x^{20}-x^{21}$ & $-2 x^2-3 x^3-3 x^4+3 x^5+3 x^6+2 x^7$ \\
$21$ & $-2+2 x^{11}+x^{19}+3 x^{20}-3 x^{21}-x^{22}$ & $-2-2 x-x^4-3 x^5+5 x^6+x^7$ \\
$22$ & $-2+2 x^{13}+x^{20}+3 x^{21}-3 x^{22}-x^{23}$ & $-5-3 x+x^3+3 x^4+2 x^5-3 x^6$ \\
$23$ & $-2+2 x^{15}+x^{21}+3 x^{22}-3 x^{23}-x^{24}$ & $-6-3 x-x^2+2 x^3+3 x^4+3 x^5-5 x^7$ \\
$24$ & $-2+2 x^{17}+x^{22}+3 x^{23}-3 x^{24}-x^{25}$ & $3+3 x-5 x^2-6 x^3-3 x^4-2 x^5+3 x^6+5 x^7$ \\
$25$ & $-2+2 x^{19}+x^{23}+3 x^{24}-3 x^{25}-x^{26}$ & $-7+3 x^2+2 x^3-3 x^4+2 x^5-2 x^6-2 x^7$ \\
$26$ & $-2+2 x^{21}+x^{24}+3 x^{25}-3 x^{26}-x^{27}$ & $-3 x+x^3-3 x^5$ \\
$27$ & $-2+2 x^{23}+x^{25}+3 x^{26}-3 x^{27}-x^{28}$ & $4 x-3 x^2-2 x^3-x^4-2 x^5-3 x^6+4 x^7$ \\
$28$ & $-2+2 x^{25}+x^{26}+3 x^{27}-3 x^{28}-x^{29}$ & $-2-2 x+2 x^2-x^3+2 x^4+x^5-5 x^7$ \\
$29$ & $-3+3 x^{27}+3 x^{28}-3 x^{29}$ & $3+3 x-3 x^3-6 x^4-3 x^5+3 x^6+3 x^7$ \\
\bottomrule
\end{longtable}
}

\section{\texorpdfstring{$U_t^{\bmod b}(x) \bmod{\Phi_b(x)}$}{Ut remainders} table}\label{appendix_ut}

{\scriptsize
\begin{longtable}{lll}
\toprule $t \bmod 3$ & $U_t^{\bmod 3}(x)$ & $U_t^{\bmod 3}(x) \bmod{\Phi_3(x)}$ \\
\midrule
$0$ & $-3 x+3 x^2$ & $-3-6 x$ \\
$1$ & $1-x$ & $1-x$ \\
$2$ & $1-x^2$ & $2+x$ \\
\midrule $t \bmod 5$ & $U_t^{\bmod 5}(x)$ & $U_t^{\bmod 5}(x) \bmod{\Phi_5(x)}$ \\ 
\midrule
$0$ & $-5 x+5 x^4$ & $-5-10 x-5 x^2-5 x^3$ \\
$1$ & $x-x^3$ & $x-x^3$ \\
$2$ & $-3-2 x+2 x^2+3 x^3$ & $-3-2 x+2 x^2+3 x^3$ \\
$3$ & $3-3 x^2$ & $3-3 x^2$ \\
$4$ & $2-2 x+3 x^2-3 x^4$ & $5+x+6 x^2+3 x^3$ \\
\midrule $t \bmod 6$ & $U_t^{\bmod 6}(x)$ & $U_t^{\bmod 6}(x) \bmod{\Phi_6(x)}$ \\ 
\midrule
$0$ & $-4 x-x^2+x^4+4 x^5$ & $5-10 x$ \\
$1$ & $1-x^4$ & $1+x$ \\
$2$ & $-1+x^2+2 x^3-2 x^5$ & $-6+3 x$ \\
$3$ & $-4 x-x^2+x^4+4 x^5$ & $5-10 x$ \\
$4$ & $1-x^4$ & $1+x$ \\
$5$ & $-1+x^2+2 x^3-2 x^5$ & $-6+3 x$ \\
\midrule $t \bmod 10$ & $U_t^{\bmod 10}(x)$ & $U_t^{\bmod 10}(x) \bmod{\Phi_{10}(x)}$ \\
\midrule
$0$ & $-4 x+x^4-x^6+4 x^9$ & $3-6 x+3 x^2-3 x^3$ \\
$1$ & $x^6-x^8$ & $-x+x^3$ \\
$2$ & $-1-2 x+2 x^3-2 x^5+2 x^7+x^8$ & $1-2 x-2 x^2+x^3$ \\
$3$ & $1-x^2+2 x^5-2 x^7$ & $-1+x^2$ \\
$4$ & $-2 x+x^2-x^4+2 x^5+2 x^7-2 x^9$ & $-3-x-2 x^2+x^3$ \\
$5$ & $-4 x+x^4-x^6+4 x^9$ & $3-6 x+3 x^2-3 x^3$ \\
$6$ & $x^6-x^8$ & $-x+x^3$ \\
$7$ & $-1-2 x+2 x^3-2 x^5+2 x^7+x^8$ & $1-2 x-2 x^2+x^3$ \\
$8$ & $1-x^2+2 x^5-2 x^7$ & $-1+x^2$ \\
$9$ & $-2 x+x^2-x^4+2 x^5+2 x^7-2 x^9$ & $-3-x-2 x^2+x^3$ \\
\midrule $t \bmod 15$ & $U_t^{\bmod 15}(x)$ & $U_t^{\bmod 15}(x) \bmod{\Phi_{15}(x)}$ \\ 
\midrule
$0$ & $-4 x+x^4-x^{11}+4 x^{14}$ & $4-3 x-4 x^2+4 x^3-3 x^4+5 x^6-4 x^7$ \\
$1$ & $x^6-x^{13}$ & $-1+x+x^4-x^5+x^6+x^7$ \\
$2$ & $-1-2 x+2 x^3-2 x^5+2 x^7+x^8$ & $-2-x+x^3+x^4-3 x^5+3 x^7$ \\
$3$ & $-2 x-x^2+2 x^5-2 x^7+x^{10}+2 x^{11}$ & $-1-4 x-x^2+x^5-2 x^6-2 x^7$ \\
$4$ & $2-2 x-x^4+2 x^7-2 x^9+x^{12}$ & $4-2 x-3 x^2+2 x^3-x^4+2 x^6-x^7$ \\
$5$ & $-2 x+2 x^4-x^6+2 x^9-2 x^{11}+x^{14}$ & $-1+x^2-x^3+x^4+x^7$ \\
$6$ & $-x+x^8+2 x^{11}-2 x^{13}$ & $-3-x^3+3 x^4-3 x^5-2 x^6+3 x^7$ \\
$7$ & $-2-2 x+x^3-x^{10}+2 x^{12}+2 x^{13}$ & $1-4 x-2 x^2+x^3-2 x^4+3 x^5-4 x^7$ \\
$8$ & $2-2 x^2+x^5-x^{12}$ & $2-x^2+x^5+x^7$ \\
$9$ & $-2 x+2 x^2-2 x^4+2 x^5+x^7-x^{14}$ & $-1-2 x+3 x^2-x^3-x^4+2 x^5-x^6+2 x^7$ \\
$10$ & $-3 x+2 x^4-2 x^6+3 x^9$ & $-3-3 x+3 x^2-3 x^3+2 x^4-5 x^6+3 x^7$ \\
$11$ & $-2 x-x^3+2 x^6-2 x^8+x^{11}+2 x^{13}$ & $4-7 x+x^3-4 x^4+4 x^5+x^6-4 x^7$ \\
$12$ & $-2 x+2 x^2-x^5+2 x^8-2 x^{10}+x^{13}$ & $1-x+2 x^2-2 x^3+x^4+x^7$ \\
$13$ & $1-2 x+2 x^6-x^7+2 x^{10}-2 x^{12}$ & $-1-2 x+2 x^2-2 x^5+2 x^6+x^7$ \\
$14$ & $-2 x+x^2-x^9+2 x^{10}+2 x^{12}-2 x^{14}$ & $-3-2 x-x^3+2 x^4-2 x^5-x^6-x^7$ \\
\midrule $t \bmod 30$ & $U_t^{\bmod 30}(x)$ & $U_t^{\bmod 30}(x) \bmod{\Phi_{30}(x)}$ \\ 
\midrule
$0$ & $-4 x+x^4-x^{26}+4 x^{29}$ & $-4-5 x+4 x^2+4 x^3+5 x^4-3 x^6-4 x^7$ \\
$1$ & $x^6-x^{28}$ & $-1-x+x^4+x^5+x^6-x^7$ \\
$2$ & $-1-2 x+2 x^3-2 x^5+2 x^7+x^8$ & $-2-3 x+3 x^3+x^4-x^5+x^7$ \\
$3$ & $-2 x-x^2+2 x^5-2 x^7+x^{10}+2 x^{11}$ & $-1-4 x-x^2+3 x^5+2 x^6-2 x^7$ \\
$4$ & $-2 x-x^4+2 x^7-2 x^9+x^{12}+2 x^{15}$ & $-4-2 x+x^2+2 x^3-x^4-2 x^6+x^7$ \\
$5$ & $-2 x-x^6+2 x^9-2 x^{11}+x^{14}+2 x^{19}$ & $3-3 x^2-3 x^3-3 x^4+3 x^7$ \\
$6$ & $-2 x-x^8+2 x^{11}-2 x^{13}+x^{16}+2 x^{23}$ & $5-3 x^3-5 x^4-5 x^5+2 x^6+5 x^7$ \\
$7$ & $-2 x-x^{10}+2 x^{13}-2 x^{15}+x^{18}+2 x^{27}$ & $1-4 x+2 x^2-x^3+2 x^4+x^5-4 x^7$ \\
$8$ & $-x^{12}+2 x^{15}-2 x^{17}+x^{20}$ & $-2+3 x^2-x^5-x^7$ \\
$9$ & $-2 x+2 x^5-x^{14}+2 x^{17}-2 x^{19}+x^{22}$ & $-1-2 x-x^2+x^3+3 x^4+2 x^5-x^6-2 x^7$ \\
$10$ & $-2 x+2 x^9-x^{16}+2 x^{19}-2 x^{21}+x^{24}$ & $1-x-x^2-x^3-2 x^4+3 x^6+x^7$ \\
$11$ & $-2 x+2 x^{13}-x^{18}+2 x^{21}-2 x^{23}+x^{26}$ & $-4-5 x+3 x^3+4 x^4+4 x^5-3 x^6-4 x^7$ \\
$12$ & $-2 x+2 x^{17}-x^{20}+2 x^{23}-2 x^{25}+x^{28}$ & $1+x-2 x^2-2 x^3-3 x^4+3 x^7$ \\
$13$ & $1-2 x+2 x^{21}-x^{22}+2 x^{25}-2 x^{27}$ & $3-2 x-2 x^2-2 x^5-2 x^6+3 x^7$ \\
$14$ & $-2 x+x^2-x^{24}+2 x^{25}+2 x^{27}-2 x^{29}$ & $5-2 x-3 x^3-2 x^4-2 x^5+3 x^6+x^7$ \\
$15$ & $-4 x+x^4-x^{26}+4 x^{29}$ & $-4-5 x+4 x^2+4 x^3+5 x^4-3 x^6-4 x^7$ \\
$16$ & $x^6-x^{28}$ & $-1-x+x^4+x^5+x^6-x^7$ \\
$17$ & $-1-2 x+2 x^3-2 x^5+2 x^7+x^8$ & $-2-3 x+3 x^3+x^4-x^5+x^7$ \\
$18$ & $-2 x-x^2+2 x^5-2 x^7+x^{10}+2 x^{11}$ & $-1-4 x-x^2+3 x^5+2 x^6-2 x^7$ \\
$19$ & $-2 x-x^4+2 x^7-2 x^9+x^{12}+2 x^{15}$ & $-4-2 x+x^2+2 x^3-x^4-2 x^6+x^7$ \\
$20$ & $-2 x-x^6+2 x^9-2 x^{11}+x^{14}+2 x^{19}$ & $3-3 x^2-3 x^3-3 x^4+3 x^7$ \\
$21$ & $-2 x-x^8+2 x^{11}-2 x^{13}+x^{16}+2 x^{23}$ & $5-3 x^3-5 x^4-5 x^5+2 x^6+5 x^7$ \\
$22$ & $-2 x-x^{10}+2 x^{13}-2 x^{15}+x^{18}+2 x^{27}$ & $1-4 x+2 x^2-x^3+2 x^4+x^5-4 x^7$ \\
$23$ & $-x^{12}+2 x^{15}-2 x^{17}+x^{20}$ & $-2+3 x^2-x^5-x^7$ \\
$24$ & $-2 x+2 x^5-x^{14}+2 x^{17}-2 x^{19}+x^{22}$ & $-1-2 x-x^2+x^3+3 x^4+2 x^5-x^6-2 x^7$ \\
$25$ & $-2 x+2 x^9-x^{16}+2 x^{19}-2 x^{21}+x^{24}$ & $1-x-x^2-x^3-2 x^4+3 x^6+x^7$ \\
$26$ & $-2 x+2 x^{13}-x^{18}+2 x^{21}-2 x^{23}+x^{26}$ & $-4-5 x+3 x^3+4 x^4+4 x^5-3 x^6-4 x^7$ \\
$27$ & $-2 x+2 x^{17}-x^{20}+2 x^{23}-2 x^{25}+x^{28}$ & $1+x-2 x^2-2 x^3-3 x^4+3 x^7$ \\
$28$ & $1-2 x+2 x^{21}-x^{22}+2 x^{25}-2 x^{27}$ & $3-2 x-2 x^2-2 x^5-2 x^6+3 x^7$ \\
$29$ & $-2 x+x^2-x^{24}+2 x^{25}+2 x^{27}-2 x^{29}$ & $5-2 x-3 x^3-2 x^4-2 x^5+3 x^6+x^7$ \\
\bottomrule
\end{longtable}
}

\section{\texorpdfstring{$W_t^{\bmod b}(x) \bmod{\Phi_b(x)}$}{Wt remainders} table}\label{appendix_wt}

{\scriptsize
\begin{longtable}{lll}
\toprule $t \bmod 3$ & $U_t^{\bmod 3}(x)$ & $U_t^{\bmod 3}(x) \bmod{\Phi_3(x)}$ \\
\midrule
$0$ & $-5 x+5 x^2$ & $-5-10 x$ \\
$1$ & $-1+x$ & $-1+x$ \\
$2$ & $3-3 x^2$ & $6+3 x$ \\
\midrule $t \bmod 5$ & $U_t^{\bmod 5}(x)$ & $U_t^{\bmod 5}(x) \bmod{\Phi_5(x)}$ \\ 
\midrule
$0$ & $-3 x+3 x^4$ & $-3-6 x-3 x^2-3 x^3$ \\
$1$ & $-x+x^3$ & $-x+x^3$ \\
$2$ & $-1-2 x+2 x^2+x^3$ & $-1-2 x+2 x^2+x^3$ \\
$3$ & $1-x^2$ & $1-x^2$ \\
$4$ & $2-2 x+x^2-x^4$ & $3-x+2 x^2+x^3$ \\
\midrule $t \bmod 6$ & $U_t^{\bmod 6}(x)$ & $U_t^{\bmod 6}(x) \bmod{\Phi_6(x)}$ \\ 
\midrule
$0$ & $-4 x+x^2-x^4+4 x^5$ & $3-6 x$ \\
$1$ & $-1+x^4$ & $-1-x$ \\
$2$ & $1-x^2+2 x^3-2 x^5$ & $-2+x$ \\
$3$ & $-4 x+x^2-x^4+4 x^5$ & $3-6 x$ \\
$4$ & $-1+x^4$ & $-1-x$ \\
$5$ & $1-x^2+2 x^3-2 x^5$ & $-2+x$ \\
\midrule $t \bmod 10$ & $U_t^{\bmod 10}(x)$ & $U_t^{\bmod 10}(x) \bmod{\Phi_{10}(x)}$ \\
\midrule
$0$ & $-4 x-x^4+x^6+4 x^9$ & $5-10 x+5 x^2-5 x^3$ \\
$1$ & $-x^6+x^8$ & $x-x^3$ \\
$2$ & $1-2 x+2 x^3-2 x^5+2 x^7-x^8$ & $3-2 x-2 x^2+3 x^3$ \\
$3$ & $-1+x^2+2 x^5-2 x^7$ & $-3+3 x^2$ \\
$4$ & $-2 x-x^2+x^4+2 x^5+2 x^7-2 x^9$ & $-5+x-6 x^2+3 x^3$ \\
$5$ & $-4 x-x^4+x^6+4 x^9$ & $5-10 x+5 x^2-5 x^3$ \\
$6$ & $-x^6+x^8$ & $x-x^3$ \\
$7$ & $1-2 x+2 x^3-2 x^5+2 x^7-x^8$ & $3-2 x-2 x^2+3 x^3$ \\
$8$ & $-1+x^2+2 x^5-2 x^7$ & $-3+3 x^2$ \\
$9$ & $-2 x-x^2+x^4+2 x^5+2 x^7-2 x^9$ & $-5+x-6 x^2+3 x^3$ \\
\midrule $t \bmod 15$ & $U_t^{\bmod 15}(x)$ & $U_t^{\bmod 15}(x) \bmod{\Phi_{15}(x)}$ \\ 
\midrule
$0$ & $-4 x-x^4+x^{11}+4 x^{14}$ & $4-5 x-4 x^2+4 x^3-5 x^4+3 x^6-4 x^7$ \\
$1$ & $-x^6+x^{13}$ & $1-x-x^4+x^5-x^6-x^7$ \\
$2$ & $1-2 x+2 x^3-2 x^5+2 x^7-x^8$ & $2-3 x+3 x^3-x^4-x^5+x^7$ \\
$3$ & $-2 x+x^2+2 x^5-2 x^7-x^{10}+2 x^{11}$ & $1-4 x+x^2+3 x^5-2 x^6-2 x^7$ \\
$4$ & $2-2 x+x^4+2 x^7-2 x^9-x^{12}$ & $4-2 x-x^2+2 x^3+x^4+2 x^6+x^7$ \\
$5$ & $-2 x+2 x^4+x^6+2 x^9-2 x^{11}-x^{14}$ & $-3+3 x^2-3 x^3+3 x^4+3 x^7$ \\
$6$ & $-3 x+3 x^8+2 x^{11}-2 x^{13}$ & $-5-3 x^3+5 x^4-5 x^5-2 x^6+5 x^7$ \\
$7$ & $-2-2 x-x^3+x^{10}+2 x^{12}+2 x^{13}$ & $-1-4 x-2 x^2-x^3-2 x^4+x^5-4 x^7$ \\
$8$ & $2-2 x^2-x^5+x^{12}$ & $2-3 x^2-x^5-x^7$ \\
$9$ & $-2 x+2 x^2-2 x^4+2 x^5-x^7+x^{14}$ & $1-2 x+x^2+x^3-3 x^4+2 x^5+x^6-2 x^7$ \\
$10$ & $-x+2 x^4-2 x^6+x^9$ & $-1-x+x^2-x^3+2 x^4-3 x^6+x^7$ \\
$11$ & $-2 x+x^3+2 x^6-2 x^8-x^{11}+2 x^{13}$ & $4-5 x+3 x^3-4 x^4+4 x^5+3 x^6-4 x^7$ \\
$12$ & $-2 x+2 x^2+x^5+2 x^8-2 x^{10}-x^{13}$ & $-1+x+2 x^2-2 x^3+3 x^4+3 x^7$ \\
$13$ & $-1-2 x+2 x^6+x^7+2 x^{10}-2 x^{12}$ & $-3-2 x+2 x^2-2 x^5+2 x^6+3 x^7$ \\
$14$ & $-2 x-x^2+x^9+2 x^{10}+2 x^{12}-2 x^{14}$ & $-5-2 x-3 x^3+2 x^4-2 x^5-3 x^6+x^7$ \\
\midrule $t \bmod 30$ & $U_t^{\bmod 30}(x)$ & $U_t^{\bmod 30}(x) \bmod{\Phi_{30}(x)}$ \\ 
\midrule
$0$ & $-4 x-x^4+x^{26}+4 x^{29}$ & $-4-3 x+4 x^2+4 x^3+3 x^4-5 x^6-4 x^7$ \\
$1$ & $-x^6+x^{28}$ & $1+x-x^4-x^5-x^6+x^7$ \\
$2$ & $1-2 x+2 x^3-2 x^5+2 x^7-x^8$ & $2-x+x^3-x^4-3 x^5+3 x^7$ \\
$3$ & $-2 x+x^2+2 x^5-2 x^7-x^{10}+2 x^{11}$ & $1-4 x+x^2+x^5+2 x^6-2 x^7$ \\
$4$ & $-2 x+x^4+2 x^7-2 x^9-x^{12}+2 x^{15}$ & $-4-2 x+3 x^2+2 x^3+x^4-2 x^6-x^7$ \\
$5$ & $-2 x+x^6+2 x^9-2 x^{11}-x^{14}+2 x^{19}$ & $1-x^2-x^3-x^4+x^7$ \\
$6$ & $-2 x+x^8+2 x^{11}-2 x^{13}-x^{16}+2 x^{23}$ & $3-x^3-3 x^4-3 x^5+2 x^6+3 x^7$ \\
$7$ & $-2 x+x^{10}+2 x^{13}-2 x^{15}-x^{18}+2 x^{27}$ & $-1-4 x+2 x^2+x^3+2 x^4+3 x^5-4 x^7$ \\
$8$ & $x^{12}+2 x^{15}-2 x^{17}-x^{20}$ & $-2+x^2+x^5+x^7$ \\
$9$ & $-2 x+2 x^5+x^{14}+2 x^{17}-2 x^{19}-x^{22}$ & $1-2 x-3 x^2-x^3+x^4+2 x^5+x^6+2 x^7$ \\
$10$ & $-2 x+2 x^9+x^{16}+2 x^{19}-2 x^{21}-x^{24}$ & $3-3 x-3 x^2-3 x^3-2 x^4+5 x^6+3 x^7$ \\
$11$ & $-2 x+2 x^{13}+x^{18}+2 x^{21}-2 x^{23}-x^{26}$ & $-4-7 x+x^3+4 x^4+4 x^5-x^6-4 x^7$ \\
$12$ & $-2 x+2 x^{17}+x^{20}+2 x^{23}-2 x^{25}-x^{28}$ & $-1-x-2 x^2-2 x^3-x^4+x^7$ \\
$13$ & $-1-2 x+2 x^{21}+x^{22}+2 x^{25}-2 x^{27}$ & $1-2 x-2 x^2-2 x^5-2 x^6+x^7$ \\
$14$ & $-2 x-x^2+x^{24}+2 x^{25}+2 x^{27}-2 x^{29}$ & $3-2 x-x^3-2 x^4-2 x^5+x^6-x^7$ \\
$15$ & $-4 x-x^4+x^{26}+4 x^{29}$ & $-4-3 x+4 x^2+4 x^3+3 x^4-5 x^6-4 x^7$ \\
$16$ & $-x^6+x^{28}$ & $1+x-x^4-x^5-x^6+x^7$ \\
$17$ & $1-2 x+2 x^3-2 x^5+2 x^7-x^8$ & $2-x+x^3-x^4-3 x^5+3 x^7$ \\
$18$ & $-2 x+x^2+2 x^5-2 x^7-x^{10}+2 x^{11}$ & $1-4 x+x^2+x^5+2 x^6-2 x^7$ \\
$19$ & $-2 x+x^4+2 x^7-2 x^9-x^{12}+2 x^{15}$ & $-4-2 x+3 x^2+2 x^3+x^4-2 x^6-x^7$ \\
$20$ & $-2 x+x^6+2 x^9-2 x^{11}-x^{14}+2 x^{19}$ & $1-x^2-x^3-x^4+x^7$ \\
$21$ & $-2 x+x^8+2 x^{11}-2 x^{13}-x^{16}+2 x^{23}$ & $3-x^3-3 x^4-3 x^5+2 x^6+3 x^7$ \\
$22$ & $-2 x+x^{10}+2 x^{13}-2 x^{15}-x^{18}+2 x^{27}$ & $-1-4 x+2 x^2+x^3+2 x^4+3 x^5-4 x^7$ \\
$23$ & $x^{12}+2 x^{15}-2 x^{17}-x^{20}$ & $-2+x^2+x^5+x^7$ \\
$24$ & $-2 x+2 x^5+x^{14}+2 x^{17}-2 x^{19}-x^{22}$ & $1-2 x-3 x^2-x^3+x^4+2 x^5+x^6+2 x^7$ \\
$25$ & $-2 x+2 x^9+x^{16}+2 x^{19}-2 x^{21}-x^{24}$ & $3-3 x-3 x^2-3 x^3-2 x^4+5 x^6+3 x^7$ \\
$26$ & $-2 x+2 x^{13}+x^{18}+2 x^{21}-2 x^{23}-x^{26}$ & $-4-7 x+x^3+4 x^4+4 x^5-x^6-4 x^7$ \\
$27$ & $-2 x+2 x^{17}+x^{20}+2 x^{23}-2 x^{25}-x^{28}$ & $-1-x-2 x^2-2 x^3-x^4+x^7$ \\
$28$ & $-1-2 x+2 x^{21}+x^{22}+2 x^{25}-2 x^{27}$ & $1-2 x-2 x^2-2 x^5-2 x^6+x^7$ \\
$29$ & $-2 x-x^2+x^{24}+2 x^{25}+2 x^{27}-2 x^{29}$ & $3-2 x-x^3-2 x^4-2 x^5+x^6-x^7$ \\
\bottomrule
\end{longtable}
}

\section{Roots of certain polynomials}\label{problematic_polynomials}

In this appendix section, we will prove that some required polynomials truly do not possess any root that is a root of unity as well, as claimed multiple times throughout Section \ref{d2_family}. The polynomials of interest shall be the four following ones:
\begin{align*}
    Z_1(x) &= 2x^{10} + x^5 + 2, & Z_2(x) &= 2x^{10} - x^5 + 2,\\
    Z_3(x) &= 2x^6 + x^3 + 2, & Z_4(x) &= 2x^6 - x^3 + 2 .
\end{align*}
First and foremost, let us define the next two additional polynomials:
\begin{align*}
    Z'(x) &= 2x^2 + x + 2, & Z''(x) &= 2x^2 - x + 2 .
\end{align*}
It is clear that the roots of $Z_1(x)$ are actually the fifth roots of the roots of $Z'(x)$, while the roots of $Z_2(x)$ are the fifth roots of the roots of $Z''(x)$. Similarly, the roots of $Z_3(x)$ represent the third roots of the roots of $Z'(x)$, while the roots of $Z_4(x)$ represent the third roots of the roots of $Z''(x)$. Thus, in order to prove that neither of the four polynomials $Z_1(x), Z_2(x), Z_3(x), Z_4(x)$ contains a root that is a root of unity, it is sufficient to show that neither of the two polynomials $Z'(x), Z''(x)$ has such a root.

Now, if the polynomial $Z'(x)$ were to have a root that is a root of unity, then this root would surely represent a primitive $b$-th root of unity for some $b \in \mathbb{N}$, which would imply that $\Phi_b(x) \mid Z'(x)$ must hold. However, it is simple to show that the polynomial $Z'(x)$ is not divisible by any cyclotomic polynomial. The same logic can be applied to the $Z''(x)$ polynomial.

It is clear that neither $Z'(x)$ nor $Z''(x)$ can be divisible by a cyclotomic polynomial whose degree is greater than two. However, given the fact that $\deg \Phi_b = \varphi(b)$ for each $b \in \mathbb{N}$, where $\varphi$ represents Euler's totient function, it is easy to deduct that only five cyclotomic polynomials have a degree that is not greater than two:
\begin{align*}
    \Phi_1(x) &= x - 1,\\
    \Phi_2(x) &= x + 1,\\
    \Phi_3(x) &= x^2 + x + 1,\\
    \Phi_4(x) &= x^2 + 1,\\
    \Phi_6(x) &= x^2 - x + 1 .
\end{align*}
It is trivial to check that neither $Z'(x)$ nor $Z''(x)$ is divisible by any of these five polynomials, which completes the proof. \hfill\qed

\end{document}